\documentclass[orivec,runningheads,envcountsame,verbatim]{llncs}

\usepackage{graphicx}
\usepackage{amsmath}
\usepackage{amssymb}

\newcommand{\quantq}{\mathbf{Q}}

\newcommand{\repnot}[2]{[#1]^{#2}}

 \newcommand{\sss} {\trianglelefteq}
 
 \newcommand{\buni} { \{ \boldsymbol{0},  \boldsymbol{1}\}^* }

\newcommand{\sigf}[3]{ \Sigma_{#1,#2,#3} }
\newcommand{\msigf}[4]{\Sigma_{#1,#2,#3}^\mathfrak{#4}}

\begin{document}             

\newcommand{\ccom}[1]{\mbox{\footnotesize{#1}}}

\setlength{\parindent}{0mm}
 \setlength{\parskip}{4pt}

\title{First-Order Concatenation Theory with Bounded Quantifiers}

\author{Lars Kristiansen\inst{1,2} \and Juvenal Murwanashyaka\inst{1}}

 \institute{
Department of Mathematics, University of Oslo, Norway
\and
    Department of Informatics, University of Oslo, Norway
}

 \maketitle                   

\newcommand{\xd}{\texttt{D}}

\newcommand{\integer}{\ensuremath{\mathbb Z}}
\newcommand{\rational}{\ensuremath{\mathbb Q}}
\newcommand{\nat}{\ensuremath{\mathbb N}}
\newcommand{\real}{\ensuremath{\mathbb R}}
\newcommand{\kleeneT}{{\mathcal T}}

\newcommand{\kleeneU}{{\mathcal U}}

\begin{abstract}
We study first-order concatenation theory with bounded quantifiers.
We give  axiomatizations with interesting properties, and we prove some normal-form results.
Finally, we prove a number of decidability and undecidability results.
\end{abstract}

\section{Introduction}

\subsection{Concatenation Theory vs.\ Number Theory}
 
First-order concatenation theory can be compared to first-order number theory, e.g.,  Peano Arithmetic or Robinson Arithmetic.
The universe of a standard structure for first-order number theory is the set of natural numbers.
The universe of a standard structure for first-order concatenation theory is the set of finite strings over some alphabet.
A first-order language for number theory normally contains two binary functions symbols. In a standard structure these 
symbols will be interpreted as addition and multiplication. A first-order language for
  concatenation theory normally contains just one binary function symbol.  In a standard structure this  symbol will be interpreted
as the operator that concatenates  two stings. A classical first-order language for concatenation theory---like e.g.\ the ones studied in Quine \cite{quine}
and Grzegorczyk \cite{grz}---contains no other non-logical symbols apart from constant symbols. 

We will stick to a version of concatenation theory 
where we have a binary alphabet consisting of the bits zero and one, and we will refer to this version as {\em bit theory}.
It is convenient to introduce and explain some  notation before we continue our discussion.

\subsection{Notation and Terminology}

We will use $\boldsymbol{0}$ and $\boldsymbol{1}$ to denote respectively the bits zero and  one,
and we use pretty standard notation when we work with bit strings: $\{ \boldsymbol{0},  \boldsymbol{1}\}^*$ denotes the set of all finite bit strings; $|\alpha|$ denotes the 
length of the bit string $\alpha$; $(\alpha)_i$ denotes the $i^{\mbox{{\scriptsize th}}}$ bit of the bit string $\alpha$; and 
$\boldsymbol{0}\boldsymbol{1}^3\boldsymbol{0}^2\boldsymbol{1}$ denotes the bit string $\boldsymbol{0}\boldsymbol{1}\boldsymbol{1}\boldsymbol{1}\boldsymbol{0}\boldsymbol{0}\boldsymbol{1}$.
The set $\{ \boldsymbol{0},  \boldsymbol{1}\}^*$
contains the empty string which we will denote $\varepsilon$.

The language $\mathcal{L}_{BT}^{-}$ of first-order bit theory consists
of the constants symbols $e,0,1$ and the binary function symbol $\circ$.
We will use $\mathfrak{B}^{-}$ to denote   the standard  $\mathcal{L}_{BT}^{-}$-structure:  
The universe of $\mathfrak{B}^{-}$ is the set $\{ \boldsymbol{0},  \boldsymbol{1}\}^*$. 
The constant symbol $0$ is interpreted as the string containing nothing but the bit $\boldsymbol{0}$, 
and the constant symbol $1$ is interpreted as the string
containing nothing but the bit $\boldsymbol{1}$, that is, $0^{\mathfrak{B}^{-}} = \boldsymbol{0}$ and  
$1^{\mathfrak{B}^{-}} = \boldsymbol{1}$. The constant symbol $e$ is interpreted as the empty string, that is, $e^{\mathfrak{B}^{-}} = \varepsilon$. 
Finally, $\circ^{\mathfrak{B}^{-}}$ is
the function that concatenates two strings 
(e.g. $\boldsymbol{0}\boldsymbol{1} \circ^{\mathfrak{B}^{-}}  \boldsymbol{0}\boldsymbol{0}\boldsymbol{0} 
=  \boldsymbol{0}\boldsymbol{1}\boldsymbol{0}\boldsymbol{0}\boldsymbol{0}$  and
$ \varepsilon \circ^{\mathfrak{B}^{-}}  \varepsilon    = \varepsilon$).

We will soon need the {\em biterals}.
For each $\alpha\in \{ \boldsymbol{0},  \boldsymbol{1}\}^*$
we define the {\em biteral} $\overline{\alpha}$ by 
$\overline{\varepsilon} = e$, $\overline{\alpha \boldsymbol{0}} = \overline{\alpha} \circ 0$ and $\overline{\alpha\boldsymbol{1}} 
= \overline{\alpha} \circ 1$. The biterals correspond to the numerals of first-order number theory:
 they serve as canonical names for the elements in the universe of the standard structure.
Note that, e.g., $((e \circ 1)\circ 1) \circ 0$ is a biteral whereas $(e\circ 1)\circ (1\circ 0)$ is not.

\subsection{Decidability and Undecidability}

At a  first glance 
the parsimonious language of  first-order bit theory does not seem very expressive, but a little bit of thought
 shows otherwise. Observe that we can encode a sequence $a_1, a_2, \ldots, a_n$ of
natural numbers by a the bit string
$$
\boldsymbol{0}\boldsymbol{0}\boldsymbol{1}\boldsymbol{0}\boldsymbol{1}^{a_1+1}\boldsymbol{0}\boldsymbol{0}
\boldsymbol{1}\boldsymbol{1}\boldsymbol{0}\boldsymbol{1}^{a_2+1}\boldsymbol{0}\boldsymbol{0}
\boldsymbol{1}^3\boldsymbol{0}\boldsymbol{1}^{a_3+1}\boldsymbol{0}\boldsymbol{0}\ldots 
\boldsymbol{0}\boldsymbol{0}\boldsymbol{1}^n\boldsymbol{0}\boldsymbol{1}^{a_n+1}\boldsymbol{0}\boldsymbol{0}\; .
$$
Furthermore, we can state that a string is a substring of another (the formula $(\exists uv)[u \circ x \circ v = y]$
holds in $\mathfrak{B}^{-}$ iff $x$ is a substring of $y$). We can state that a string contain only ones 
(the formula $\neg(\exists uv)[u \circ 0 \circ v = x]$ holds in  $\mathfrak{B}^{-}$ iff $x$ contains only ones, and so does the the formula $x\circ 1= 1\circ x$).
We can state that a bit string does not contain two consecutive occurrences of zeros, and so on. If we proceed 
along these lines, we can come up with a formula $\phi(x,y,z)$ such that 
$\phi(\overline{\boldsymbol{1}^{i}}, \overline{\boldsymbol{1}^{a}},\overline{\alpha})$ is true  in $\mathfrak{B}$ iff
$\alpha$ encodes a sequence of natural numbers where the  $i^{\mbox{{\scriptsize th}}}$ element is $a$. The reader interested in the details should consult 
 Section 8 of Leary \& Kristiansen \cite{leary}. 

First-order bit theory is indeed expressive enough to code and decode sequences of natural numbers, and then
it should be no surprise that the following theorem holds.
E.g.,  it is straightforward
to prove the theorem by induction over a Kleene-recursive definition  of $f$.

\begin{theorem}[Definability of Computable Functions]\label{bsigsafcomp}
For any (partially) computable function $f:\nat^n \rightarrow \nat$ there exists an $\mathcal{L}_{BT}^{-}$-formula $\phi(x_1,\ldots , x_n,y)$ such that
$$
f(a_1, \ldots, a_n)= b \;\; \Leftrightarrow 
\;\; \mathfrak{B}^{-} \models \phi(\overline{\boldsymbol{1}^{a_1}},\ldots , \overline{\boldsymbol{1}^{a_n}},\overline{\boldsymbol{1}^{b}})\; .
$$
Moreover, given a  definition of the function, we can compute the formula. 
\end{theorem}

This theorem implies that it is undecidable if a  $\mathcal{L}_{BT}^{-}$-sentence is true in the standard structure $\mathfrak{B}^{-}$.
It is of course also undecidable if a sentence  of first-order number theory is true in its standard structure $\mathfrak{N}$.
Indeed, due to the negative solution of Hilbert's $10^{\mbox{{\scriptsize th}}}$ problem (the Davis-Putnam-Robinson-Matiyasevich Theorem), 
it is even undecidable if a sentence of the form $$\exists x_1, \ldots , x_n[ \ s=t \ ]$$ is true in $\mathfrak{N}$. The bit-theoretic
version of Hilbert's $10^{\mbox{{\scriptsize th}}}$ problem turned out to have a positive solution. The next theorem is 
just a reformulation of a result of Makanin \cite{makanin}.

\begin{theorem}[Makanin]\label{baaaaasigsafcomp}
It is decidable if an $\mathcal{L}_{BT}^{-}$-sentence of the form $$\exists x_1, \ldots , x_n[ \ s=t \ ]$$ is true in $\mathfrak{B}^{-}$.
\end{theorem}

Hence, it is in general undecidable if a sentence is true in $\mathfrak{B}^{-}$, but it is decidable if a
sentence of the form $\exists x_1, \ldots x_n[s=t]$ is true in $\mathfrak{B}^{-}$. 
This raises the question: where do we find the border between the decidable
and the undecidable? For which subclasses of formulas can we, or can we not,
decide truth in the standard structure? Such a question yields an obvious motivation for
introducing bounded quantifiers similar to those we know 
from number theory. Once the bounded quantifiers are there, a number of other questions will knock at the door.

\subsection{Bounded Quantifiers and $\Sigma$-formulas}

The first-order language $\mathcal{L}_{BT}$ is $\mathcal{L}_{BT}^{-}$ extended by  a binary relation symbol $\sqsubseteq$. We introduce the {\em bounded
existential quantifier} $(\exists x \sqsubseteq  t)\phi$ and 
 {\em bounded
universal quantifier} $(\forall x \sqsubseteq  t)\phi$ as shorthand notations for respectively
$$
\exists x [ \ x \sqsubseteq t \, \wedge \, \phi \  ]\;\;\;\mbox{ and } 
\;\;\; \forall x [ \ x \sqsubseteq t \, \rightarrow \, \phi \ ] \; .
$$
Next we define the $\Sigma$-formulas inductively
by
\begin{itemize}
\item $\phi$ and $\neg \phi$ are $\Sigma$-formulas if $\phi$ is  of the form $s\sqsubseteq t$, or of the form 
$s = t$, where $s$ and $t$ are terms
\item $(\phi \vee \psi)$ and $(\phi \wedge \psi)$ are  $\Sigma$-formulas if $\phi$ and $\psi$ are $\Sigma$-formulas
\item  $(\exists x  \sqsubseteq t) \phi$ and $(\forall x  \sqsubseteq t) \phi$ and $(\exists x) \phi$ are 
 $\Sigma$-formulas if $\phi$ is a $\Sigma$-formula, $t$ is a term and $x$ is a variable not occurring in $t$.
\end{itemize}
A {\em purely existential formula} is  a $\Sigma$-formula that  does not contain  bounded universal quantifiers.

We assume that the reader notes the similarities with first-order number theory.
The formulas that correspond to $\Sigma$-formulas in number theory are often called $\Sigma_1$-formulas or $\Sigma^0_1$-formulas.
Now,  in contrast to in number theory, it is not clear how the relation symbol that defines the bounded quantifiers should be
interpreted. In the standard model for number theory the relation symbol  
should obviously be interpreted as the standard ordering relation of the natural numbers. 
There does not seem to be any other natural options.
We can chose between the less-than or less-than-or-equal-to relation over the natural numbers, and that is it.
In bit theory a number of essentially different interpretations will make sense. We might interpret $x \sqsubseteq y$ as
\begin{itemize}
\item  $x$ is a substring  $y$
\item  $x$ is a prefix of  $y$
\item  $x$ is shorter than  $y$.
\end{itemize}
The three interpretations listed above are the ones 
that will be investigated in the current paper, but
there are definitely other interesting options.
We might  e.g.\ interpret $\sqsubseteq$ as  the subword relation investigated in Halfon et. al.\ \cite{halfon}, or as
a lexicographical ordering relation.

\subsection{More on Notation and Terminology}

 We reserve the letter $\mathfrak{B}$ for the structure where $\sqsubseteq$ is the substring relation, that is,
the $\mathcal{L}_{BT}$-structure $\mathfrak{B}$ is the extension of $\mathfrak{B}^{-}$ where  $\alpha \sqsubseteq^\mathfrak{B} \beta$ holds 
iff there exists $\gamma,\delta\in  \buni$ such that $\gamma\alpha\delta=\beta$.

 The $\mathcal{L}_{BT}$-structure $\mathfrak{D}$ is the same structure as $\mathfrak{B}$ with one exception: the relation
$\alpha \sqsubseteq^\mathfrak{D} \beta$ holds iff $\alpha$ is a prefix of $\beta$, that is, iff 
 there exists  $\gamma \in \buni$ such that $\alpha\gamma= \beta$. 

 The $\mathcal{L}_{BT}$-structure $\mathfrak{F}$ is the same structure as $\mathfrak{B}$   with one exception: the relation
$\alpha \sqsubseteq^\mathfrak{F} \beta$ holds iff the number of bits in $\alpha$ is less than or equal to the 
number of bits in $\beta$, that is,
iff $|\alpha|\leq |\beta|$.

To improve the readability we will use the  symbol $\preceq$
in place of the symbol $\sqsubseteq$ in formulas that are meant to be interpreted in the structure  $\mathfrak{D}$. 
Thus,  $x \preceq y$
should be read as ``$x$ is a prefix of $y$''. 
Similarly, we will use $\sss$ in formulas that are meant to be interpreted in $\mathfrak{F}$,
and $x\sss y$ should be read as ``$x$ is shorter than $y$''.
We will continue to use the symbol $\sqsubseteq$ in formulas that are meant to be interpreted in  $\mathfrak{B}$. Thus,  $x \sqsubseteq y$
should be read as ``$x$ is a substring of $y$''.

 We may 
 skip the operator $\circ$ in first-order formulas and simply write $st$ in place of $s\circ t$.
Furthermore,
we will  occasionally contract quantifiers and 
write, e.g., $\forall x y \sqsubseteq z[\phi] \ $ in place of $(\forall x \sqsubseteq z)(\forall  y \sqsubseteq z) \phi $, and 
for $\sim \, \in \! \{\preceq, \sqsubseteq,\sss,  =\}$, we  will normally write $s\not \sim t$ in place of $\neg s\sim t$.

Recall that a {\em sentence} is a formula with no free variables.

\subsection{References and Related Research}

Formal concatenation theory can be traced back to Tarski \cite{tarski} and Hermes \cite{hermes} (see \cite{quinerew} for an English review).
Quine \cite{quine} and Corcoran et al.\  \cite{corcoran} are also papers on the subject that may have some historical interest.
 A rather recent line of research in concatenation theory has focused on interpretablity between weak first-order theories and  essential undecidability.
Grzegorczyk \cite{grz}, Grzegorczyk \& Zdanowski \cite{zd}, Svejdar \cite{sved}, Visser \cite{visser}, Horihata \cite{hori} and  Higuchi \& Horihata  \cite{hig}
belong to this line. Another recent line of research has focused on word equations and formal languages. Papers in this line, like e.g., 
Karhum\"aki et al.\ \cite{karh},  Ganesh et al.\ \cite{ganesh}, Halfon et al.\ \cite{halfon} and  Day et al.\ \cite{day}, are in general  oriented towards theoretical computer science.
Both lines are related to our research: the former to the first-order theories we present in Section \ref{secrefen}, the latter to the results we present in last two  sections of the paper. 
More on the history of concatenation theory can be found in  Visser \cite{visser}.

The present paper is a significantly extended an improved version of the conference paper Kristiansen \& Murwanashyaka \cite{cie18}.
We assume that the reader is familiar with the basics of mathematical logic and computability theory. An introduction can be found Leary \& Kristiansen \cite{leary}
(Chapter 8 contains some material on concatenation theory). Some familiarity with first-order arithmetic will probably also be 
 beneficial to the reader. An introduction
can be found in Hajek \& Pudlak \cite{hajek}.

\section{$\Sigma$-Complete Axiomatizations}

\label{secrefen}

Once the bounded quantifiers are present, it is natural to ask if we can find
neat and natural $\Sigma$-complete axiomatisations of bit theory. Corresponding axiomatizations of number
theory, e.g. Robinson Arithmetic,  have been of great importance in logic. 
In this section we will give
finite sets of axioms that are $\Sigma$-complete with respect to
our different structures ($\mathfrak{B}$, $\mathfrak{D}$ and $\mathfrak{F}$).
These axiomatizations might serve as base theories which  can be extended with, e.g., collection principles or
induction schemes of the form
$$
\big( \ \phi(e) \; \wedge \; \forall x [  \ \phi(x) \rightarrow (\phi(x0) \wedge \phi(x1)) \ ]\ \big)\;\; \rightarrow \;\;
\forall x [ \phi(x) ]\; .
$$

\subsection{The Theory $B^{-}$}

$B^{-}$ is a  $\mathcal{L}_{BT}^{-}$-theory.
All the first-order theories we will present contain the axioms  of $B^{-}$.

\begin{definition}
The first-order theory $B^{-}$ contains the following four non-logical axioms:
\begin{enumerate} 
\item $\forall x[ \ x= e  x \wedge x=x  e \ ] $
\item $\forall x y z [ \ (x  y)  z = x  (y   z) \ ]$
\item $\forall x y[ \ ( x \neq y) \to 
( \ ( x   0 \neq y   0)  \wedge ( x   1 \neq y   1) \ ) \ ] $
\item $ \forall x y [ \ x   0 \neq y   1 \ ] $
\end{enumerate} 
We will use $B^{-}_i$ to refer to the $i^{\mbox{{\scriptsize th}}}$ axiom of $B^{-}$.
\end{definition}

\begin{lemma}\label{seksommorgen}
Let $\alpha,\beta\in \{ \boldsymbol{0},  \boldsymbol{1}\}^*$. Then
$B^{-}\vdash \overline{\alpha} \circ \overline{\beta} = \overline{\alpha\beta}$.
\end{lemma}

\begin{proof}
We prove the lemma by induction on the length of $\beta$, and we consider the following cases: $\beta \equiv \varepsilon$, 
$\beta \equiv \gamma\boldsymbol{0}$ and $\beta \equiv \gamma\boldsymbol{1}$.

Assume $\beta \equiv \varepsilon$. Then $\overline{\beta}\equiv e$, and the lemma holds by the axiom $B^{-}_1$.
Assume $\beta \equiv \gamma\boldsymbol{0}$. By the induction hypothesis, we have
$B^{-}\vdash \overline{\alpha} \circ \overline{\gamma} = \overline{\alpha\gamma}$.
Thus, we also have
$$B^{-}\vdash (\overline{\alpha} \circ \overline{\gamma}) \circ 0 = \overline{\alpha\gamma} \circ 0\; .$$
By $B^{-}_2$, we have
$$B^{-}\vdash \overline{\alpha} \circ (\overline{\gamma} \circ 0) = \overline{\alpha\gamma}  \circ 0 \; .$$
Thus, the lemma holds as $\overline{\beta} \equiv \overline{\gamma} \circ 0$ and 
$\overline{\alpha\beta} \equiv \overline{\alpha\gamma\boldsymbol{0}} \equiv \overline{\alpha\gamma} \circ 0$.
The case  $\beta \equiv \gamma\boldsymbol{1}$ is similar.
\qed
\end{proof}

\begin{lemma}\label{adagenetter}
For any  $\mathcal{L}_{BT}$-term $t$ there is $\alpha\in \{ \boldsymbol{0},  \boldsymbol{1}\}^*$ such that
$B^{-}\vdash t=\overline{\alpha}$.
\end{lemma}

\begin{proof} 
We will use induction on the structure of $t$.
Assume $t\equiv e$. Obviously,  $B^{-}\vdash e=e$. Thus, we have $B^{-}\vdash e =\overline{\varepsilon}$ as $\overline{\varepsilon}\equiv e$. 
Assume $t\equiv 0$. By $B^{-}_1$, we have $B^{-}\vdash 0 = e\circ 0$. Thus, we have  $B^{-}\vdash 0 = \overline{\boldsymbol{0}}$ as 
$\overline{\boldsymbol{0}} = e\circ 0$. The case  $t\equiv 1$ is similar to the case  $t\equiv 0$. Finally,
assume $t\equiv t_1\circ t_2$. By induction hypothesis we have $\alpha_1, \alpha_2 \in \{ \boldsymbol{0},  \boldsymbol{1}\}^*$
such that $B^{-}\vdash t_1=\overline{\alpha_1}$ and $B^{-}\vdash t_2=\overline{\alpha_2}$. By Lemma \ref{seksommorgen}, we have
$B^{-}\vdash \overline{\alpha_1} \circ \overline{\alpha_2} = \overline{\alpha_1\alpha_2}$.
Thus, $B^{-}\vdash t =\overline{\alpha_1\alpha_2}$.
\qed
\end{proof}

\begin{lemma} \label{gammelbtre}
$  B^{-}_1, B^{-}_2, B^{-}_4 
\vdash \forall x[ \ x 0 \neq e \wedge   x1 \neq e \ ]$.
\end{lemma}

\begin{proof}
We reason in an arbitrary model for $\{ B^{-}_1, B^{-}_2, B^{-}_4 \}$.
Let $x$ be an arbitrary element in the universe.
Assume $x0=e$. Then $1(x0)=1e$. By $B^{-}_1$, we have $1(x0)=1$.
By $B^{-}_2$, we have $(1x)0=1$. By $B^{-}_1$, we have  $(1x)0=e1$.
This contradicts $B^{-}_4$. This proves that $x0\neq e$.
A symmetric argument shows that  $x1\neq e$. 
\qed
\end{proof}

\begin{lemma}\label{siste}
Let $\alpha,\beta\in \{ \boldsymbol{0},  \boldsymbol{1}\}^*$ and $\alpha\neq \beta$. Then
$B^{-}\vdash \overline{\alpha}  \neq \overline{\beta}$.
\end{lemma}

\begin{proof}
We prove the lemma by induction on the natural number $\min(|\alpha|,|\beta|)$.
Assume $\min(|\alpha|,|\beta|)=0$. Then,  either $\alpha$ or $\beta$ will be the empty string $\varepsilon$.
Furthermore, it cannot be the case that both $\alpha$ or $\beta$ are the empty string as $\alpha\neq \beta$.
By Lemma \ref{gammelbtre}, we have $B^{-}\vdash \overline{\alpha}  \neq \overline{\beta}$.

Assume $\min(|\alpha|,|\beta|)>0$. Then we have $\alpha\equiv \alpha'a$
and $\beta\equiv \beta'b$ where $\alpha,\beta\in \{ \boldsymbol{0},  \boldsymbol{1}\}^*$ and 
$a,b\in \{ \boldsymbol{0},  \boldsymbol{1}\}$. The proof splits into two cases: (i) $a$ and $b$ are equal, and (ii)
$a$ and $b$ are different. {\sl Case (i):} The induction hypothesis yields 
$ B^{-}\vdash \overline{\alpha'}  \neq \overline{\beta'}$, and the lemma follows by $B^{-}_3$. 
{\sl Case (ii):} The lemma follows straightaway by $B^{-}_4$ (we do not need
the induction hypothesis).
\qed
\end{proof}

We leave the proof of the next theorem to the reader (see the proof of Theorem \ref{bsigcomp}).

\begin{theorem}\label{purebsigcomp}
For any purely existential sentence $\phi$, we have
$$
 \mathfrak{B}^{-} \models \phi  \; \Rightarrow \; B^{-}\vdash \phi \; .
$$
\end{theorem}

\subsection{The Structure $\mathfrak{B}$}

\begin{definition}
The first-order theory $B$ contains the following eleven non-logical axioms:
\begin{itemize}
\item[-] the first four axioms are the axioms of $B^{-}$ 
\item[5.] $\forall x [ \ x \sqsubseteq e \leftrightarrow x=e \ ] $
\item[6.] $\forall x[ \ x \sqsubseteq 0 \leftrightarrow (x=e \vee x = 0) \ ] $
\item[7.] $\forall x [ \ x \sqsubseteq 1 \leftrightarrow (x=e \vee x = 1) \ ] $
\item[8.] 
$ \forall x y [ \ x \sqsubseteq 0  y   0 \leftrightarrow (x = 0  y   0 \vee x \sqsubseteq 0  y  \vee x \sqsubseteq y  0 ) \ ]     $
\item[9.] $\forall x y [ \ x \sqsubseteq 0  y   1 \leftrightarrow (x = 0  y   1 \vee x \sqsubseteq 0  y  \vee x \sqsubseteq y  1 ) \ ]     $
\item[10.]
 $ \forall x  y [ \ x \sqsubseteq 1  y   0 \leftrightarrow (x = 1  y   0 \vee x \sqsubseteq 1  y  \vee x \sqsubseteq y  0 ) \ ]     $
\item[11.] 
$ \forall x y [ \ x \sqsubseteq 1  y   1 \leftrightarrow (x = 1  y   1 \vee x \sqsubseteq 1  y  \vee x \sqsubseteq y  1 ) \ ]     $
\end{itemize} 
We will use $B_i$ to refer to the $i^{\mbox{{\scriptsize th}}}$ axiom of $B$.
\end{definition}

\begin{lemma}\label{billy}
Let $\alpha,\beta\in \{ \boldsymbol{0},  \boldsymbol{1}\}^*$ and $\alpha \sqsubseteq^\mathfrak{B} \beta$ (i.e. $\alpha$ is a substring
of $\beta$). Then
$B\vdash \overline{\alpha}  \sqsubseteq \overline{\beta}$. 
\end{lemma}

\begin{proof}
We prove this lemma by 
induction on the length of $\beta$. The proof splits into the cases $\beta\equiv \varepsilon$, $\beta\equiv \boldsymbol{0}$,
$\beta\equiv \boldsymbol{1}$, $\beta\equiv \boldsymbol{0}\gamma \boldsymbol{0}$, $\beta\equiv \boldsymbol{0}\gamma \boldsymbol{1}$, $\beta\equiv \boldsymbol{1}\gamma \boldsymbol{0}$
 and $\beta\equiv \boldsymbol{1}\gamma \boldsymbol{1}$.

In the cases when $\beta$ is empty or just contains a
single bit, the lemma holds by $B_5$, $B_6$ or $B_7$.

Assume $\beta\equiv \boldsymbol{0}\gamma\boldsymbol{0}$. The proof splits into the   cases
$$ \mbox{(i) $\alpha= \boldsymbol{0}\gamma\boldsymbol{0}$,
(ii) $\alpha \sqsubseteq^\mathfrak{B}  \boldsymbol{0}\gamma$ and (iii) $\alpha \sqsubseteq^\mathfrak{B}  \gamma\boldsymbol{0}$.}$$
In case (i), we  have $B \vdash \overline{\alpha}  = \overline{\beta}$ by logical axioms. By $B_8$, we have $B\vdash \overline{\alpha}  \sqsubseteq \overline{\beta}$. In case (ii), the induction hypothesis yields 
$B\vdash \overline{\alpha}  \sqsubseteq \overline{ \boldsymbol{0}\gamma}$. By $B_8$, we have $B\vdash \overline{\alpha}  \sqsubseteq \overline{\beta}$.
In case (iii), the induction hypothesis yields 
$B\vdash \overline{\alpha}  \sqsubseteq \overline{ \gamma\boldsymbol{0}}$. 
By $B_8$, we have $B\vdash \overline{\alpha} \sqsubseteq \overline{\beta}$.

The cases  $\beta\equiv \boldsymbol{0}\gamma\boldsymbol{1}$, $\beta\equiv \boldsymbol{1}\gamma\boldsymbol{0}$
and $\beta\equiv \boldsymbol{1}\gamma\boldsymbol{1}$ are similar to the case  $\beta\equiv \boldsymbol{0}\gamma\boldsymbol{0}$,
use $B_9$, $B_{10}$ and $B_{11}$, respectively, in place of $B_8$. 
\qed
\end{proof}

\begin{lemma}\label{bbilly}
Let $\alpha,\beta\in \{ \boldsymbol{0},  \boldsymbol{1}\}^*$ 
and $\alpha \not\sqsubseteq^\mathfrak{B} \beta$ (i.e. $\alpha$ is a not substring
of $\beta$). Then
$B\vdash \overline{\alpha}  \not\sqsubseteq \overline{\beta}$.
\end{lemma}

\begin{proof}
The proof of this lemma is symmetric to the proof of Lemma \ref{billy}, and we omit the details.
\qed
\end{proof}

\begin{lemma} \label{varmerekordikoben}
Let $\phi(x)$ be an $\mathcal{L}_{BT}$-formula such that 
\begin{align*}
\mathfrak{B} \models \phi(\overline{\alpha}) \; \Rightarrow \; B \vdash \phi(\overline{\alpha})  \tag{*}
\end{align*}
for any $\alpha\in \{ \boldsymbol{0},  \boldsymbol{1}\}^*$.
Then we have
\[ \mathfrak{B} \models (\forall x  \sqsubseteq \overline{\alpha})\phi(x)  \; \Rightarrow \; 
B \vdash (\forall x  \sqsubseteq \overline{\alpha})\phi(x)  \]
 for any  $\alpha\in \{ \boldsymbol{0},  \boldsymbol{1}\}^*$.
\end{lemma}

\begin{proof}
We proceed by induction on the length of $\alpha$. We will consider the cases  $\alpha\equiv \varepsilon$, $\alpha\equiv \boldsymbol{0}$,
$\alpha\equiv \boldsymbol{1}$, $\alpha\equiv \boldsymbol{0}\beta \boldsymbol{0}$, $\alpha\equiv \boldsymbol{0}\beta \boldsymbol{1}$, $\alpha\equiv \boldsymbol{1}\beta \boldsymbol{0}$
 and $\alpha\equiv \boldsymbol{1}\beta \boldsymbol{1}$.

Let $\alpha \equiv \varepsilon$. Assume $\mathfrak{B} \models (\forall x  \sqsubseteq e)\phi(x)$.
Then $\mathfrak{B} \models \phi(e)$. By (*), we have $B\vdash \phi(e)$. By $B_5$, we have
$B \vdash (\forall x  \sqsubseteq e)\phi(x)$.

Let $\alpha \equiv \boldsymbol{0}$. Assume $\mathfrak{B} \models (\forall x  \sqsubseteq  \overline{\boldsymbol{0}})\phi(x)$.
Then $\mathfrak{B} \models \phi(\overline{\varepsilon}) \wedge \phi(\overline{\boldsymbol{0}})$.
By (*), we have $B\vdash \phi(\overline{\varepsilon}) \wedge \phi(\overline{\boldsymbol{0}})$.
By $B_6$, we have
$B \vdash (\forall x  \sqsubseteq \overline{\boldsymbol{0}})\phi(x)$.
The case  $\alpha \equiv \boldsymbol{1}$ is similar  to the case  $\alpha \equiv \boldsymbol{0}$.
Use  $B_7$ in place of $B_6$.

Let $\alpha \equiv \boldsymbol{0}\beta\boldsymbol{0}$. Assume
 $\mathfrak{B} \models (\forall x  \sqsubseteq \boldsymbol{0}\beta\boldsymbol{0} )\phi(x)$. 
Then
$$
 \mathfrak{B} \models \phi(\boldsymbol{0}\beta\boldsymbol{0}) \; \wedge \;   (\forall x  \sqsubseteq  \boldsymbol{0}\beta)\phi(x) \; \wedge \; 
(\forall x  \sqsubseteq   \beta\boldsymbol{0})\phi(x)\; .
$$
By (*) and the induction hypothesis, we have
$$
B \vdash \phi(\boldsymbol{0}\beta\boldsymbol{0}) \; \wedge \;   (\forall x  \sqsubseteq  \boldsymbol{0}\beta)\phi(x) \; \wedge \; 
(\forall x  \sqsubseteq   \beta\boldsymbol{0})\phi(x) \; .
$$
By $B_8$, we have $B\vdash (\forall x  \sqsubseteq \boldsymbol{0}\beta\boldsymbol{0} )\phi(x)$.

The case  $\alpha \equiv \boldsymbol{0}\beta\boldsymbol{1}$, the case  $\alpha \equiv \boldsymbol{1}\beta\boldsymbol{0}$ and 
the case  $\alpha \equiv \boldsymbol{1}\beta\boldsymbol{1}$ are handled similarly using $B_9$, $B_{10}$ and $B_{11}$, respectively, in place of $B_8$. 
\qed
\end{proof}

\begin{theorem}[$\Sigma$-completeness of $B$]\label{bsigcomp}
For any $\Sigma$-sentence $\phi$, we have
$$
 \mathfrak{B} \models \phi  \; \Rightarrow \; B\vdash \phi \; .
$$
\end{theorem}

\begin{proof} 
We prove the theorem by induction on the structure of the $\Sigma$-sentence $\phi$. The base cases are
 $\phi\equiv s=t$,   $\phi\equiv s\neq t$,
  $\phi\equiv s \sqsubseteq t$ and   $\phi\equiv s \not\sqsubseteq t$ (where $s$ and $t$ are variable free). 
We attend to the case $\phi\equiv s \sqsubseteq t$. 
So  assume  $\phi\equiv s \sqsubseteq t$. Furthermore, assume  $\mathfrak{B} \models s \sqsubseteq t$.
By Lemma \ref{adagenetter}, we have $\alpha,\beta \in  \{ \boldsymbol{0},  \boldsymbol{1}\}^*$ 
such that 
\begin{align*}
B\vdash s= \overline{\alpha} \; \wedge \; t= \overline{\beta}\; . \tag{*} 
\end{align*}
By the Soundness Theorem for first-order logic, we have $\alpha \sqsubseteq^{\mathfrak{B}} \beta$.
By Lemma \ref{billy}, we have $B\vdash  \overline{\alpha} \sqsubseteq \overline{\beta}$.
By (*), we have $B\vdash  s \sqsubseteq t$. This proves that the theorem holds when $\phi\equiv s \sqsubseteq t$.
The  cases  $\phi\equiv s=t$,   $\phi\equiv s\neq t$ and   $\phi\equiv s \not\sqsubseteq t$ are similar.
Use Lemma \ref{siste} in place of  Lemma \ref{billy} when $\phi\equiv s\neq t$.
Use Lemma \ref{bbilly} in place of  Lemma \ref{billy} when $\phi\equiv s \not\sqsubseteq t$.

We turn to the inductive cases. Let $\phi \equiv (\psi \wedge \xi)$. Assume $\mathfrak{B} \models \psi \wedge \xi$.
Then we  have $\mathfrak{B} \models \psi$ and $\mathfrak{B} \models \xi$. By the induction hypothesis, we have
$B\vdash  \psi$ and $B \vdash \xi$. Thus, $B\vdash  \psi \wedge \xi$. The case  $\phi \equiv (\psi \vee \xi)$ is similar.

Let $\phi \equiv (\exists x)\psi(x)$. Assume $\mathfrak{B} \models (\exists x)\psi(x)$. Then we have
$\mathfrak{B} \models \psi(\overline{\alpha})$ for some $\alpha \in  \{ \boldsymbol{0},  \boldsymbol{1}\}^*$. 
Our induction hypothesis yields $B\vdash \psi(\overline{\alpha})$, and then we also have 
 $B\vdash (\exists x)\psi(x)$. 

Let $\phi \equiv (\forall x\sqsubseteq t)\psi(x)$.
 Assume $\mathfrak{B} \models (\forall x\sqsubseteq t)\psi(x)$. By Lemma \ref{adagenetter}, we have 
$\beta \in  \{ \boldsymbol{0},  \boldsymbol{1}\}^*$ such that
\begin{align*}
B\vdash  t = \overline{\beta}  \tag{$\dagger$}
\end{align*}
By the Soundness Theorem of first-order logic, we have 
\begin{align*}
\mathfrak{B} \models (\forall x\sqsubseteq  \overline{\beta})\psi(x) \tag{$\ddagger$}
\end{align*}
Our induction hypothesis yields
\begin{align*}
\mathfrak{B} \models \psi(\overline{\alpha})  \;\; \Rightarrow B\vdash \psi(\overline{\alpha})   \tag{IH}
\end{align*}
for all $\alpha \in  \{ \boldsymbol{0},  \boldsymbol{1}\}^*$. 
By (IH), ($\ddagger$) and Lemma \ref{varmerekordikoben}, we have $B\vdash (\forall x\sqsubseteq  \overline{\beta})\psi(x)$.
Finally, by  ($\dagger$), we have $B\vdash (\forall x\sqsubseteq  t)\psi(x)$. This completes the case where $\phi$ is of the form 
$(\forall x\sqsubseteq t)\psi(x)$. We leave the case $\phi \equiv (\exists x\sqsubseteq t)\psi(x)$ to the reader.
\qed
\end{proof}

\subsection{The Structure $\mathfrak{D}$}

\begin{definition}
The first-order theory $D$ contains the following seven non-logical axioms:
\begin{enumerate} 
\item[-] the first four axioms are the axioms of $B^{-}$
\item[5.] $\forall x[ \ x \preceq e \leftrightarrow x=e \ ]$
\item[6.] $ \forall x y [ \  x \preceq y  0 \leftrightarrow 
( x = y  0 \vee x \preceq y) \ ]$ 
\item[7.] $\forall x y [ \  x \preceq y  1  \leftrightarrow 
(x = y  1 \vee x \preceq y) \ ]$
\end{enumerate}
We will use $D_i$ to refer to the $i^{\mbox{{\scriptsize th}}}$ axiom of $D$.
\end{definition}

\begin{lemma}\label{dbilly}
Let $\alpha,\beta\in \{ \boldsymbol{0},  \boldsymbol{1}\}^*$ and $\alpha \preceq^\mathfrak{D} \beta$ (i.e. $\alpha$ is a prefix
of $\beta$). Then
$D\vdash \overline{\alpha}  \preceq \overline{\beta}$. 
\end{lemma}

\begin{proof} 
This proof is symmetric to the proof of the next lemma. We omit the details.
\qed
\end{proof}

\begin{lemma}\label{dbbilly}
Let $\alpha,\beta\in \{ \boldsymbol{0},  \boldsymbol{1}\}^*$ 
and $\alpha \not\preceq^\mathfrak{D} \beta$ (i.e. $\alpha$ is a not prefix
of $\beta$). Then
$D\vdash \overline{\alpha}  \not\preceq \overline{\beta}$.
\end{lemma}

\begin{proof} 
We prove the lemma by induction on the length of $\beta$, and we consider the cases  $\beta\equiv \varepsilon$, $\beta\equiv \gamma\boldsymbol{0}$ and  $\beta\equiv \gamma\boldsymbol{1}$.

Assume $\beta$ is empty. Then  $\alpha$ is not empty. By Lemma \ref{siste}, we have $D\vdash \overline{\alpha} \neq e$. By axiom $D_5$, we have 
$D\vdash \overline{\alpha} \not\preceq e$. Thus,  $D\vdash \overline{\alpha} \not\preceq \overline{\beta}$.
Assume $\beta\equiv \gamma\boldsymbol{0}$. Then we have $\alpha \neq \gamma\boldsymbol{0}$ and $\alpha \not\preceq^\mathfrak{D} \gamma$.
By our induction hypothesis, we have $D\vdash \overline{\alpha}  \not\preceq \overline{\gamma}$.  By Lemma \ref{siste},
we have $D\vdash \overline{\alpha} \neq \overline{\gamma\boldsymbol{0}}$. By axiom $D_6$, we have $D\vdash \overline{\alpha} \not\preceq \overline{\gamma\boldsymbol{0}}$.
The case where $\beta$ is of the form $\gamma\boldsymbol{1}$ is similar. Apply axiom $D_7$ in place of $D_6$.
\qed
\end{proof}

\begin{lemma} \label{dvarmerekordikoben}
Let $\phi(x)$ be an $\mathcal{L}_{BT}$-formula such that 
\begin{align*}
\mathfrak{D} \models \phi(\overline{\alpha}) \; \Rightarrow \; D \vdash \phi(\overline{\alpha})  \tag{*}
\end{align*}
for any $\alpha\in \{ \boldsymbol{0},  \boldsymbol{1}\}^*$.
Then we have
\[ \mathfrak{D} \models (\forall x  \preceq \overline{\alpha})\phi(x)  \; \Rightarrow \; 
D \vdash (\forall x  \preceq \overline{\alpha})\phi(x)  \]
 for any  $\alpha\in \{ \boldsymbol{0},  \boldsymbol{1}\}^*$.
\end{lemma}

\begin{proof} 
We prove the lemma by induction on the length of $\alpha$, and  we consider the cases   $\alpha\equiv \varepsilon$, $\alpha\equiv \beta\boldsymbol{0}$ and  $\alpha\equiv \beta\boldsymbol{1}$.

Let $\alpha\equiv \varepsilon$. Assume $\mathfrak{D} \models (\forall x  \preceq e)\phi(x)$. Then we have $\mathfrak{D} \models \phi(e)$.
By (*), we have $D\vdash \phi(e)$. By $D_5$, we have $D\vdash (\forall x  \preceq e)\phi(x)$. 

Let $\alpha\equiv \beta\boldsymbol{0}$. Assume $\mathfrak{D} \models (\forall x  \preceq \overline{\beta\boldsymbol{0}})\phi(x)$.
Then we have $\mathfrak{D} \models (\forall x  \preceq \overline{\beta})\phi(x)$ and $\mathfrak{D} \models \phi(\overline{\beta\boldsymbol{0}})$.
By our induction hypothesis we have $D\vdash (\forall x  \preceq \overline{\beta})\phi(x)$. By (*), we have $D\vdash \phi(\overline{\beta\boldsymbol{0}})$.
By axiom $D_6$ we have $D \vdash (\forall x  \preceq \overline{\beta\boldsymbol{0}})\phi(x)$. 

The case $\alpha\equiv \beta\boldsymbol{1}$ is similar to the preceding case. Use $D_7$ in place of $D_6$.
 \qed
\end{proof}

\begin{theorem}[$\Sigma$-completeness of $D$]\label{dsigcomp}
For any $\Sigma$-sentence $\phi$, we have
$$
  \mathfrak{D} \models \phi \; \Rightarrow \;  D\vdash \phi\; .
$$
\end{theorem}

\begin{proof} 
Given the lemmas above, we can more or less just repeat the proof of Theorem \ref{bsigcomp}.
Use Lemma \ref{dbilly} in place of Lemma \ref{billy}, Lemma \ref{dbbilly} in place of Lemma \ref{bbilly}
and Lemma \ref{dvarmerekordikoben} in place of Lemma \ref{varmerekordikoben}.
\qed
\end{proof}

\subsection{The Structure $\mathfrak{F}$}

\begin{definition}
The first-order theory $F$ contains the following eleven non-logical axioms:
\begin{enumerate} 
\item[-] the first four axioms are the  axioms of $B^{-}$
\item[5.] $\forall x[ \ e \sss x  \ ]$
\item[6.] $\forall x[ \ x \sss e \rightarrow x=e  \ ]$
\item[7.] $\forall x y [ \  x0 \sss y0   \leftrightarrow x\sss y \ ]$
\item[8.] $\forall x y [ \  x0 \sss y1   \leftrightarrow x\sss y \ ]$
\item[9.] $\forall x y [ \  x1 \sss y0   \leftrightarrow x\sss y \ ]$
\item[10.] $\forall x y [ \  x1 \sss y1   \leftrightarrow x\sss y \ ]$.
\item[11.] $\forall x [ \  x= e \; \vee \;  \exists y [ \ x=y0 \, \vee \, x=y1  \ ] \ ]$.
\end{enumerate}
We will use $F_i$ to refer to the $i^{\mbox{{\scriptsize th}}}$ axiom of $F$.
\end{definition}

\begin{lemma}\label{fbilly}
Let $\alpha,\beta\in \{ \boldsymbol{0},  \boldsymbol{1}\}^*$ and $\alpha \sss^\mathfrak{F} \beta$ (i.e.
 $|\alpha| \leq |\beta|$). Then $F\vdash \overline{\alpha}  \sss \overline{\beta}$. 
\end{lemma}

\begin{proof} 
We prove this lemma by induction on the length of $\alpha$, and we consider the cases
$\alpha\equiv \varepsilon$, $\alpha\equiv \alpha' \boldsymbol{0}$ and $\alpha\equiv \alpha' \boldsymbol{1}$.

If $\alpha\equiv \varepsilon$, we have 
$F\vdash \overline{\alpha}\sss \overline{\beta}$ by $F_5$. Let $\alpha\equiv \alpha' \boldsymbol{0}$.
Since $|\alpha|\leq |\beta|$, we have $b\in \{ \boldsymbol{0},  \boldsymbol{1}\}$ and $\beta'\in \{ \boldsymbol{0},  \boldsymbol{1}\}^*$
such that $\beta = \beta'b$. The induction hypothesis yields $F\vdash  \overline{\alpha' }\sss \overline{\beta '}$. If $b= \boldsymbol{0}$,
we have $F\vdash \overline{\alpha}\sss \overline{\beta}$ by $F_7$.  If $b= \boldsymbol{1}$,
we have $F\vdash \overline{\alpha}\sss \overline{\beta}$ by $F_8$. This proves the case  $\alpha\equiv\alpha' \boldsymbol{0}$.
The proof when  $\alpha\equiv\alpha' \boldsymbol{1}$ is similar. Use $F_9$ and $F_{10}$, respectively,  in place of 
$F_7$ and $F_{8}$.
\qed
\end{proof}

\begin{lemma}\label{fbbilly}
Let $\alpha,\beta\in \{ \boldsymbol{0},  \boldsymbol{1}\}^*$ 
and $\alpha \not\sss^\mathfrak{F} \beta$  (i.e.
 $|\alpha| > |\beta|$). Then
$F\vdash \overline{\alpha}  \not\sss \overline{\beta}$.
\end{lemma}

\begin{proof} 
We prove this lemma  by induction on the length of $\beta$, and we consider the cases $\beta\equiv \varepsilon$,
$\beta\equiv \beta'\boldsymbol{0}$ and $\beta\equiv \beta'\boldsymbol{1}$.

Assume $\beta\equiv \varepsilon$.
Since $|\alpha|\not\leq |\beta|$, we have  $\alpha'\in \{ \boldsymbol{0},  \boldsymbol{1}\}^*$
such that $\alpha = \alpha'\boldsymbol{0}$ or  $\alpha = \alpha '\boldsymbol{1}$. We can w.l.o.g.\
say that $\alpha = \alpha'\boldsymbol{0}$. By Lemma \ref{gammelbtre}, we have $F\vdash \overline{\alpha}\neq e$.
By $F_6$, we have $F\vdash \overline{\alpha}\not\sss \overline{\beta}$.
Assume $\beta\equiv \beta'\boldsymbol{0}$. Since $|\alpha|\not\leq |\beta|$, we have 
$b\in \{ \boldsymbol{0},  \boldsymbol{1}\}$ and $\alpha'\in \{ \boldsymbol{0},  \boldsymbol{1}\}^*$
such that $\alpha \equiv \alpha'b$. By our induction hypothesis, we have $F\vdash \overline{\alpha'}\not\sss \overline{\beta'}$.
 If $b= \boldsymbol{0}$,
we have $F\vdash \overline{\alpha}\not\sss \overline{\beta}$ by $F_7$.  If $b= \boldsymbol{1}$,
we have $F\vdash \overline{\alpha}\not\sss \overline{\beta}$ by $F_9$. The case $\beta\equiv \beta'\boldsymbol{1}$ is symmetric
to the case $\beta\equiv \beta'\boldsymbol{0}$. Use $F_8$ in place of $F_7$, and use $F_{10}$ in place of $F_9$.
\qed
\end{proof}

\newcommand{\bsint}[1]{ [ \varepsilon \ldots   #1 ]}

It is convenient to introduce some new notation before we state our next lemma:
For any $\alpha \in \buni$, let 
$$\bsint{\alpha} = \{ \ \beta \mid \mbox{$\beta\in \buni$ and $\beta \sss^{\mathfrak{F}} \alpha$} \}\; .$$

\begin{lemma}\label{fcases}
We have
$$
F\vdash \forall x [ \ x\sss \overline{\alpha} \; \rightarrow \bigvee_{\beta\in \bsint{\alpha}} x= \overline{\beta} \ ]
$$
for any  $\alpha\in \buni$.
\end{lemma}

\begin{proof}
We prove this lemma by induction on the length of $\alpha$. The base case is $\alpha\equiv \varepsilon$.
The inductive cases are $\alpha\equiv \gamma\boldsymbol{0}$ and $\alpha\equiv \gamma\boldsymbol{1}$.

First we deal with case $\alpha\equiv \varepsilon$. The axiom $F_6$ says that
$$
\forall x[ \ x \sss e \rightarrow x=e \ ] \; .
$$
Thus, we have 
$$
F\vdash \forall x [ \ x\sss \overline{\alpha} \; \rightarrow \bigvee_{\beta\in \bsint{\alpha}} x= \overline{\beta} \ ]
$$
straightaway as $\overline{\alpha}$ is $e$, the set $\bsint{\alpha}$ is the singleton set $\{\varepsilon \}$ and $\overline{\varepsilon}$
is $e$.

We will now turn to the inductive case $\alpha\equiv \gamma\boldsymbol{0}$.
In this case 
it is sufficient to prove 
\renewcommand{\arraystretch}{1.6}
 $$  \begin{array}{ccl}
 (1) & \;\;\;  & F\vdash \forall x [ \ x=e  \; \rightarrow  \; ( \ x \sss \overline{\gamma \boldsymbol{0}} \; \rightarrow \; 
{\displaystyle \bigvee_{\beta\in \bsint{\gamma\boldsymbol{0}} } x= \overline{\beta} }  \ ) \ ]            \\
 (2) & \;\;\;  & F\vdash \forall x y [ \ x=y0  \; \rightarrow  \; ( \ x \sss \overline{\gamma \boldsymbol{0}} \; \rightarrow \;
{\displaystyle \bigvee_{\beta\in \bsint{\gamma\boldsymbol{0}} } x= \overline{\beta} }  \ ) \ ]             \\
 (3) & \;\;\;  & F\vdash  \forall x y[ \ x=y1  \; \rightarrow  \; ( \ x \sss \overline{\gamma \boldsymbol{0}} \; \rightarrow \;
{\displaystyle \bigvee_{\beta\in \bsint{\gamma\boldsymbol{0}} } x= \overline{\beta} } \ ) \ ] \; .            \\
\end{array}$$
\renewcommand{\arraystretch}{1.0}
as it follows from (1), (2), (3) and the axiom  $F_{11}$ that
$$
F\vdash  \forall x [ \ x \sss \overline{\gamma \boldsymbol{0}} \; \rightarrow \;
\bigvee_{\beta\in \bsint{\gamma\boldsymbol{0}} } x= \overline{\beta} \ ] \; . $$
Our induction hypothesis yields 
\begin{align*}
F\vdash \forall x [ \ x\sss \overline{\gamma} \; \rightarrow \bigvee_{\beta\in \bsint{\gamma}} x= \overline{\beta} \ ] \; .       \tag{IH}
\end{align*}

It is trivial that (1) holds. This is a logical truth that holds in any model. 
We do not need any of our non-logical axioms to prove (1). 
Let us turn to the proof of (2). We reason in an arbitrary model for $F$.
Assume $x=y0$ and $x\sss \overline{\gamma \boldsymbol{0}}$. 
We need to argue that 
\begin{align*}
 \bigvee_{\beta\in \bsint{\gamma\boldsymbol{0}} } x= \overline{\beta} \tag{$\dagger$}
\end{align*}
holds in the model.
It is obvious that we have $y0 \sss \overline{\gamma \boldsymbol{0}}$. By $F_7$, we have
$y \sss \overline{\gamma}$. By (IH), we have $$\bigvee_{\beta\in \bsint{\gamma}} y= \overline{\beta}\; .$$
Thus, we also have 
$$\bigvee_{\beta\in \bsint{\gamma}}   y0 = \overline{\beta}0   $$
as $y= \overline{\beta} \rightarrow  y0 = \overline{\beta}0$.
Furthermore, since $x=y0$ and $       \overline{\beta \boldsymbol{0}} \equiv \overline{\beta}0$, we have
\begin{align*}
\bigvee_{\beta\in \bsint{\gamma}}   x =  \overline{\beta \boldsymbol{0}} \; .
 \tag{$\ddagger$}
\end{align*}
Finally, we observe that ($\ddagger$) implies ($\dagger$). This proves (2).

The proof of (3) is symmetric to the proof of (2). Use the axiom $F_9$ in place of $F_7$.
This completes the proof for the  case $\alpha\equiv \gamma\boldsymbol{0}$.
The case $\alpha\equiv \gamma\boldsymbol{1}$ is symmetric.
Use the axioms  $F_8$ and $F_{10}$, respectively, in place of $F_7$ and $F_{9}$.
\qed
\end{proof}

\begin{lemma} \label{fvarmerekordikoben}
Let $\phi(x)$ be an $\mathcal{L}_{BT}$-formula such that 
\begin{align*}
\mathfrak{F} \models \phi(\overline{\alpha}) \; \Rightarrow \; F \vdash \phi(\overline{\alpha})  \tag{*}
\end{align*}
for any $\alpha\in \buni$.
Then we have
\[ \mathfrak{F} \models (\forall x  \sss \overline{\alpha})\phi(x)  \; \Rightarrow \; 
F \vdash (\forall x  \sss \overline{\alpha})\phi(x)  \]
 for any  $\alpha\in \buni$.
\end{lemma}

\begin{proof}
Assume $\mathfrak{F} \models (\forall x  \sss \overline{\alpha})\phi(x)$.
Then, we have $$\mathfrak{F} \models \bigwedge_{\beta\in \bsint{\alpha}} \phi(\overline{\beta}) \; . $$ 
By (*), we have
$$F \vdash  \bigwedge_{\beta\in \bsint{\alpha}} \phi(\overline{\beta}) \; . $$ 
By Lemma \ref{fcases}, we have $F \vdash (\forall x  \sss \overline{\alpha})\phi(x)$.
\qed
\end{proof}

\begin{theorem}[$\Sigma$-completeness of $F$]\label{ffsigcomp}
For any $\Sigma$-sentence $\phi$, we have
$$
  \mathfrak{F} \models \phi \; \Rightarrow \;  F\vdash \phi\; .
$$
\end{theorem}

\begin{proof} 
Given the lemmas above, we can more or less just repeat the proof of Theorem \ref{bsigcomp}.
Use Lemma \ref{fbilly} in place of Lemma \ref{billy}, Lemma \ref{fbbilly} in place of Lemma \ref{bbilly}
and Lemma \ref{fvarmerekordikoben} in place of Lemma \ref{varmerekordikoben}.
\qed
\end{proof}

Both $B$ and $D$ are open theories (all the axioms are purely universal statements) whereas $F$ is not.
The axiom $F_{11}$ contains an existential quantifier.
Can we find a purely universal set of  axioms that is $\Sigma$-complete with respect to the model  $\mathfrak{F}$?
Yes, we can. We can regard Lemma  \ref{fcases} as an axiom scheme. Then we do not need $F_{11}$ to achieve $\Sigma$-completeness.

\begin{definition}
Let $F'$ be the
first-order theory $F$ where the axiom $F_{11}$ is replaced by the scheme
\begin{align*}
\;\;\;\;\;\;\;\;\;\;\;\;\;\;\;\;\;\;\;\; \forall x [ \ x\sss \overline{\alpha} \; \rightarrow \bigvee_{\beta\in \bsint{\alpha}} x= \overline{\beta} \ ]\; .
\tag{for   $\alpha\in \buni$}
\end{align*}
\end{definition}

\begin{theorem}[$\Sigma$-completeness of $F'$]\label{uendffsigcomp}
For any $\Sigma$-sentence $\phi$, we have
$$
  \mathfrak{F} \models \phi \; \Rightarrow \;  F' \vdash \phi\; .
$$
\end{theorem}

\begin{proof} 
Proceed as in the proof of Theorem \ref{ffsigcomp}. Since the axiom scheme is present, $F_{11}$ will not be needed anymore.
\qed
\end{proof}

Now, $F'$ is an open theory, but in contrast to $B$ and $D$, it is not finite.

\paragraph{Conjecture:} There is no finite open set of axioms that is $\Sigma$-complete with
respect to the structure $\mathfrak{F}$.

\section{Normal Form Theorems}

After we have endowed bit theory with bounded quantifiers, it becomes natural to search
for normal forms and see if we can prove normal form theorems similar to the ones
 we know from number theory.

Some  lemmas (\ref{sulten}, \ref{famanda}, \ref{samanda} ) in this section can   be found elsewhere, e.g., in 
B\"uchi \& Senger \cite{bs}, Senger \cite{senger} and Karhum\"aki et al.\ \cite{karh}. We have given complete proofs below in order
to make our paper self-contained (the proofs we give do not differ essentially from those given in B\"uchi \& Senger \cite{bs}).

The next lemma shows that conjunctions of equations can be replaced by one equation.

\begin{lemma}\label{sulten}
Let $s_{1}, s_{2}, t_{1}, t_{2}$ be $\mathcal{L}_{BT}$-terms. 
We have 
$$\mathfrak{B}^{-} \models  (s_{1} = t_{1} \wedge s_{2} = t_{2}) \; \leftrightarrow \;
 s_{1}  0  s_{2}  s_{1}  1  s_{2} = t_{1}  0  t_{2}  t_{1}  1  t_{2}\; .$$ 
\end{lemma}

\begin{proof}
Let $\alpha_1,\alpha_2,\beta_1,\beta_2\in\buni$.
Assume $\alpha_1  \boldsymbol{0}  \alpha_2  \alpha_1  \boldsymbol{1}  \alpha_2 = \beta_{1}  \boldsymbol{0}  \beta_{2}  \beta_{1}  \boldsymbol{1}  \beta_{2}$.
Then $|\alpha_1  \boldsymbol{0}  \alpha_2| = |\beta_{1}  \boldsymbol{0}  \beta_{2}|$ and $|\alpha_1  \boldsymbol{1} \ \alpha_2| = |\beta_{1}  \boldsymbol{1}  \alpha_2|$.
The proof splits into the two cases $|\alpha_1|=|\beta_1|$ and $|\alpha_1|\neq |\beta_1|$. In the case when $|\alpha_1|=|\beta_1|$, we obviously have $\alpha_1=\beta_1$ and $\alpha_2=\beta_2$.
Assume  $|\alpha_1|\neq |\beta_1|$. We can w.l.o.g. assume that $|\alpha_1|< |\beta_1|$. This implies that 
$$  \boldsymbol{0} \; = \; (\alpha_1  \boldsymbol{0}  \alpha_2)_{|\alpha_1| + 1}  \; = \; (t)_{|\alpha_1| + 1} \; = \; (\alpha_1  \boldsymbol{1}  \alpha_2)_{|\alpha_1| + 1} \; = \;  \boldsymbol{1}\; .$$
This is a contradiction. This proves  the implication from the right to the 
left. The converse implication is obvious.
\qed
\end{proof}

The next lemma shows that disjunctions of equations can be replaced by one equation at the price of some more
existential quantifiers.

\begin{lemma}\label{firecola}
Let $s_{1},s_{1},t_{1},t_{2}$  be $\mathcal{L}_{BT }$-terms.         
There exist $\mathcal{L}_{BT }$-terms $s ,t$ and variables $v_1,\ldots, v_k$ such that 
$$\mathfrak{D} \models (s_{1}\preceq t_{1} \vee s_{2} \preceq t_{2}) \leftrightarrow 
\exists v_{1} \ldots  v_{k}[s = t ]\; .$$
\end{lemma}

\begin{proof}
 Let $x_1,\ldots , x_6$ be variables that do not occur in any of the terms $s_1,s_2,t_1,t_2$.
It is not very hard to see that the formula 
$s_{1} \preceq t_{1} \vee  s_{2} \preceq t_{2}$
is $\mathfrak{D}$-equivalent  to the formula 
\begin{align*}
\exists x_{1} \ldots  x_{6} &
[ \ s_{1} = x_{1}  x_{2} \; \wedge \; t_{1} = x_{1}  x_{3} \; \wedge \; \\ & 
 \ s_{2} = x_{4}  x_{5} \; \wedge \;  t_{2} = x_{4}  x_{6} \; \wedge \; 
( x_{2} = e \; \vee \; x_{5} = e ) \ ] \; . \tag{*}
\end{align*}
We will show that the disjunction in (*), that is  $x_{2} = e  \vee  x_{5} = e$, can be replaced by
a formula
$$
x_2x_5 = x_5x_2 \; \wedge \; \exists y_1\ldots y_4[ \ v_1 = v^\prime_1 \; \wedge \; v_2 = v^\prime_2 \; \wedge \; v_3 = v^\prime_3 \ ]
$$
where $v_1, v_2, v_3,  v^\prime_1 , v^\prime_2 , v^\prime_3$ are terms. Thus, by Lemma \ref{sulten} which allow us to merge conjunctions of equations,
(*) will be equivalent to a formula of the form
$$
\exists x_{1} \ldots  x_{6}y_1\ldots y_4[ \ s=t \ ]
$$
and our proof will be complete.

Let $\psi(u,w)$ be the formula
\begin{align*}
  \exists y_1  y_2  y_3  y_4 [\
 y_1y_2 = 0 \; \wedge \; y_3y_4 = 1 
 \; \wedge \; 
 u y_1  w y_2 & = w   y_2  u y_1  \\ &  \; \wedge \;   
u y_3  w y_4  = w y_4  u y_3 \ ] \; .
\end{align*}

We claim that 
\begin{align*}
\mathfrak{B}^{-} \models   (  u = e \; \vee\; w = e  )  \; \leftrightarrow \; 
(  uw = wu   \; \wedge\; \psi(u,w)  )\; . \tag{**}
\end{align*}

We prove (**). Assume that $u = e  \vee w = e$ (we reason in $\mathfrak{B}^{-}$). Let us say that $u = e$ (the case when $w = e$ is symmetric).
It is obvious that we have $uw = wu$. Moreover,  $\psi(u,w)$ holds with $y_1=y_3=e$, $y_2=0$ and $y_4=1$.
This prove the left-right implication of (**). 

To see that the converse implication holds, assume that
$\neg(u = e  \vee w = e)$, that is, both $u$ and $w$ are different from the empty string. Furthermore,
assume that $uw = wu$. We will argue that $\psi(u,w)$ does not hold: Since $uw = wu$ and both $u$ and $w$
contain at least one bit, it is  either the case that 0 is the last bit of both strings, or it is that case
that 1 is the last bit of both strings. If $0$ is the last bit of both, the two equations
$u y_3  w y_4  = w y_4  u y_3$ and  $y_3y_4 = 1$ cannot be satisfied simultaneously.
If $1$ is the last bit of both, the two equations  $u y_1  w y_2  = w   y_2  u y_1$
 and   $y_1y_2 = 0$ cannot be satisfied simultaneously. Hence we conclude that $\psi(u,w)$ does not hold.
This completes the proof of (**).

As explained above, our lemma follows from (*) and (**) by Lemma \ref{sulten}.
\qed
\end{proof}

Karhum\"aki et al.\ \cite{karh}
prove that the next lemma indeed holds with $k=2$.

\begin{lemma}\label{famanda}
Let $s_{1}, s_{2}, t_{1}, t_{2}$ be $\mathcal{L}_{BT}$-terms. 
There exist $\mathcal{L}_{BT}$-terms $s,t$ and variables $v_{1},\ldots ,v_{k}$ such that 
 $$ \mathfrak{B}^{-} \models (s_{1} = t_{1} \vee s_{2} = t_{2}) \leftrightarrow 
\exists v_{1}\ldots  v_{k}[ s=t] \; . $$
\end{lemma}

\begin{proof}
Observe that  $s_{1} = t_{1} \vee s_{2} = t_{2}$ is $\mathfrak{D}$-equivalent  to
$$  (s_{1} \preceq t_{1} \wedge t_{1} \preceq s_{1} ) \vee 
(s_{2} \preceq t_{2} \wedge t_{2} \preceq s_{2} )$$
which again is  (logically) equivalent to
\begin{align*}
 (s_{1} \preceq t_{1} \; \vee \;  s_{2} \preceq t_{2} ) \; \wedge  \;
  (s_{1} \preceq t_{1} \;  \vee \;  & t_{2} \preceq s_{2} ) \; \wedge \; \\ & 
 \  ( t_{1} \preceq s_{1} \; \vee \; s_{2} \preceq t_{2} ) \; \wedge \;
( t_{1} \preceq s_{1} \; \vee \;  t_{2} \preceq s_{2} )\; .
\end{align*}
By Lemma \ref{sulten} and Lemma \ref{firecola}, we have terms $s,t$ and variables $ v_{1}\ldots  v_{k}$ such that
$$ \mathfrak{D} \models (s_{1} = t_{1} \vee s_{2} = t_{2}) \leftrightarrow 
\exists v_{1}\ldots  v_{k}[ s=t] \; . $$
Thus, the lemma holds as we are dealing with a $\mathcal{L}_{BT}^{-}$-formula.
\qed
\end{proof}

\begin{lemma}\label{samanda}
Let $s_{1}, t_{1}$ be $\mathcal{L}_{BT}$-terms. 
There exist $\mathcal{L}_{BT}$-terms $s,t$ and variables $v_{1},\ldots ,v_{k}$ such that 
$$ \mathfrak{B}^{-} \models  s_{1} \neq t_{1}  \leftrightarrow 
\exists v_{1} \ldots  v_{k} [ s=t ] \; .$$
\end{lemma}

\begin{proof}
Observe that the formula  $s\neq t$ is $\mathfrak{B}^{-}$-equivalent  
 to the formula
\begin{align*}\exists x y z [ \ s= t   0    x  \; \vee \; s= t   1    x &  \; \vee \;
 t= s   0    x \vee t= s   1    x  \; \vee \; \\ &
( s= x    1   y \; \wedge \;   t= x    0   z ) \; \vee \;   ( s= x    0   y \; \wedge \;  t= x    1   z )  \ ]\; .
\end{align*}
Thus,  the lemma follows from Lemma \ref{famanda} and Lemma \ref{sulten}.
\qed
\end{proof}

\begin{lemma}\label{martin}
Let $s_1, t_1$ be $\mathcal{L}_{BT}$-terms. There exist $\mathcal{L}_{BT}$-terms $s,t$ and variables $v_{1},\ldots, v_{k}$ such that 
 $$ \mathfrak{D} \models  s_{1} \not\preceq t_{1}  \leftrightarrow 
\exists v_{1} \ldots  v_{k} [s=t] \; .$$
\end{lemma}

\begin{proof}
The formula $ s_{1} \not\preceq t_{1}$ is $\mathfrak{D}$-equivalent   to the formula 
\begin{align*}
( \ \exists u  [t_1u = s_1] \; \wedge \;  t_1 \neq s_1 \  ) \; \vee \; 
\exists x y z [ \ 
 (t_1 = x    0    y &  \; \wedge \; s_1 =x    1    z  ) \; \vee \; \\ & 
(t_1 = x    1    y \; \wedge \; s_1 =x    0    z  ) \ ]\;  .
\end{align*}
Thus,  the lemma follows from the preceding lemmas.
\qed
\end{proof}

\begin{theorem}[Normal Form Theorem for $\mathfrak{D}$] \label{bispegaard}
Any $\Sigma$-formula $\phi$ is $\mathfrak{D}$-equivalent   to a $\mathcal{L}_{BT}$-formula  of the form 
$$  (\quantq_{1}^{t_{1}} v_{1} )\ldots (\quantq_{m}^{t_{m}} v_{m} ) \,  s=t$$
where $t_{1},..,t_{m}$ are $\mathcal{L}_{BT}$-terms and 
$\quantq_{j}^{t_{j}} v_{j} \in \lbrace \exists v_{j}, \exists v_{j} \preceq t_{j} ,  \forall v_{j} \preceq t_{j} \rbrace$ 
for  $j=1,\ldots , m$. Moreover, if $\phi$ is a purely existential formula, then  $\quantq_{j}^{t_{j}} v_{j}$ is $\exists v_j$.
\end{theorem}

\begin{proof}
We proceed by induction on the structure of the $\Sigma$-formula $\phi$ (throughout the proof we  reason in the structure $\mathfrak{D}$).
Assume $\phi \equiv s\preceq t$. Then $\phi$ is equivalent to a formula of the form $\exists x[ sx=t]$ and the theorem holds.
Assume $\phi \equiv s\not\preceq t$. Then the theorem holds by Lemma \ref{martin}. Assume $\phi \equiv s\neq t$.
Then the theorem holds by Lemma \ref{samanda}. The theorem holds trivially when $\phi$ is of the form $s=t$.

Suppose $\phi$ is of the form $\psi \wedge \xi$. By our induction hypothesis, we have formulas
$$   (\quantq_{1}^{t_{1}} x_{1})\ldots (\quantq_{k}^{t_{k}} x_{k}) \, s_{1}=t_{1} \;\;\; \mbox{ and } \;\;\;
   (\quantq_{1}^{s_{1}} y_{1})\ldots (\quantq_{m}^{s_{m}} y_{m}) \, s_{2}=t_{2} $$
which are equivalent  to respectively $\psi$ and $\xi$. 
Thus, $\phi$ is equivalent to a formula of the form
$(\quantq_{1}^{t_{1}} x_{1})\ldots (\quantq_{k}^{t_{k}} x_{k})(\quantq_{1}^{s_{1}} y_{1})\ldots (\quantq_{m}^{s_{m}} y_{m})  
( s_{1}=t_{1} \wedge s_{2}=t_{2}) \; .$
By Lemma \ref{sulten}, we have a formula  of the desired form which is equivalent to $\phi$.
The case when $\phi$ is of the form $\psi \vee \xi$ is similar. Use Lemma \ref{famanda} in place of Lemma \ref{sulten}.

The theorem follows trivially from the induction hypothesis when $\phi$ is of one of the forms $(\exists v) \psi$,
$(\forall v \preceq t) \psi$ and $(\exists v \preceq t) \psi$.

If $\phi$ is a purely existential formula, there will be no bounded
universal quantifiers among $(\quantq_{1}^{t_{1}} v_{1} )\ldots (\quantq_{m}^{t_{m}} v_{m} )$.
Thus,
 $\phi$ is equivalent to a formula of the form
$$
\exists v_1 \ldots v_k[  \ s_1\preceq t_1 \wedge \ldots \wedge s_\ell \preceq t_\ell \wedge s=t \ ]
$$
which again is equivalent to a formula of the form
$$
\exists v_1 \ldots v_kx_1\ldots x_\ell[ \ s_1 x_1 = t_1 \wedge \ldots \wedge s_\ell x_\ell = t_\ell \wedge s=t \ ]\; .
$$
By Lemma \ref{sulten}, we can conclude that  any purely existential
formula is equivalent to a formula
of the form $\exists v_1\ldots  v_{m}   [  s=t  ]$.
\qed
\end{proof}

\begin{theorem}[Normal Form Theorem for  $\mathfrak{F}$] \label{fffbispegaard}
Any $\Sigma$-formula $\phi$ is $\mathfrak{F}$-equivalent   to a $\mathcal{L}_{BT}$-formula  of   the form 
$$(\exists v_0)(\quantq_{1} v_1 \sss t_1 )\ldots (\quantq_{m} v_{m} \sss t_m) \, s=t$$
where
$\quantq_{j} \in \lbrace  \exists  ,  \forall  \rbrace$ for $j= 1, \ldots ,m$.  
Moreover, if $\phi$ is a purely existential formula, then $\phi$ is equivalent to a formula of the form
$$\exists v_0\ldots  v_{m} [ \ v_1\sss t_1 \; \wedge  \; \ldots  \; \wedge \; v_m \sss t_m \;  \wedge \;  s=t \ ]\; .$$
\end{theorem}

\begin{proof}
We prove the theorem 
 by induction on the structure of the $\Sigma$-formula $\phi$ (throughout the proof we  reason in the structure $\mathfrak{F}$).
Assume $\phi \equiv s\sss t$. Then $\phi$ is equivalent to a formula of the form $\exists x\sss t [x=s]$, and the theorem holds.
Assume $\phi \equiv s\not\sss t$. Then $\phi$ is equivalent to $t\circ 0 \sss s$, and thus also equivalent to a formula of the form
$\exists x\sss s [x=t\circ 0]$. Hence the theorem holds. Furthermore,
 the theorem holds by Lemma \ref{samanda}  when $\phi \equiv s\neq t$, and  the theorem holds trivially when $\phi \equiv s=t$.

The inductive cases  $\phi \equiv \psi \wedge \xi$ and $\phi \equiv \psi \vee \xi$
are similar to the corresponding cases in the proof of Theorem \ref{bispegaard} (normal form theorem for $\mathfrak{D}$).

The cases  $\phi\equiv (\exists x)\psi$, $ \phi\equiv (\exists x\sss t)\psi$ and $\phi\equiv (\forall x\sss t)\psi$ are
are easy to deal with.
A formula of the form 
\begin{itemize}
\item $(\forall x \sss t)(\exists y)\psi$
 is equivalent  to a formula of the form $(\exists z)(\forall x \sss t)(\exists y \sss z)\psi$
\item $(\exists x \sss t)(\exists y)\psi$ is equivalent
  to a formula of  the form $(\exists y) (\exists x \sss t)\psi$
\item $(\exists x)(\exists y)\psi$ is equivalent
  to a formula of  the form $(\exists z)(\exists x \sss z)(\exists y \sss z)\psi$.
\end{itemize}
Thus,
any $\Sigma$-formula  is $\mathfrak{F}$-equivalent  to a $\Sigma$-formula that
contains maximum one unbounded existential quantifier.

If $\phi$ is a purely existential formula, then $\quantq_{1},\ldots, \quantq_{m}$  will all be existential quantifiers. Thus, it is easy
to see that any purely existential formula is equivalent to a formula of the form
$$\exists v_0\ldots  v_{m} [ \ v_1\sss t_1 \; \wedge  \; \ldots  \; \wedge \; v_m \sss t_m \;  \wedge \;  s=t \ ]\; .$$
\qed
\end{proof}

It is not true that any purely existential formula is $\mathfrak{F}$-equivalent to a formula of the form $\exists x_1\ldots x_n [s=t]$.
This follows  from the results in Karhum\"aki et al.\ \cite{karh}.
E.g., a formula like $x\sss y \wedge y\sss x$ which states that the length of $x$ equals the length of $y$, is not $\mathfrak{F}$-equivalent
to a formula of the form $\exists x_1\ldots x_n [s=t]$. See Example 27 in Section 6 of \cite{karh}.

\begin{lemma}\label{tyggegummi}
Let $s_1,t_1$ be $\mathcal{L}_{BT}$-terms. There exist $\mathcal{L}_{BT}$-terms $s,t$ and variables $v_{1},\ldots, v_k$ such that 
$$\mathfrak{B} \models  s_1 \not\sqsubseteq t_1 \leftrightarrow 
\forall v_{1} \sqsubseteq t_1 \exists v_{2}\ldots v_{k}
[ \ s = t \ ] \; .$$
\end{lemma}

\begin{proof}
Observe that
$s_1 \not\sqsubseteq t_1$ is $\mathfrak{B}$-equivalent  to 
$(\forall v  \sqsubseteq t_1)\alpha$ where $\alpha$ is
\begin{multline*}
\exists x [ \ t_1x = v s_1 \, \wedge \,  x\neq  e \ ]\;\vee \; \exists x y z [ \ 
(t_1 = x0y \, \wedge \, vs_1 = x1z) \; \vee \; \\ (t_1 = x1y \, \wedge \, vs_1 = x0z)   \  ]  \; .
\end{multline*}
If we let $vs_1 \preceq t_1$ abbreviate $\exists x[v s_1 x = t_1]$, then $\alpha$ can be written as
$vs_1 \not\preceq t_1$.
Thus, the lemma follows by Lemma \ref{martin}.
\qed
\end{proof}

\begin{theorem}[Normal Form Theorem for $\mathfrak{B}$] \label{bbbbispegaard}
Any $\Sigma$-formula $\phi$ is $\mathfrak{B}$-equivalent in  to a $\mathcal{L}_{BT}$-formula  of   the form 
$$(\exists v_0)(\quantq_{1} v_1 \sqsubseteq t_1 )\ldots (\quantq_{m} v_{m} \sqsubseteq t_m) \, s=t$$
where
$\quantq_{j} \in \lbrace  \exists  ,  \forall  \rbrace$ for $j= 1, \ldots , m$.  
\end{theorem}

\begin{proof}
Proceed by induction on the structure of the $\Sigma$-formula $\phi$. This proof is similar to the proof of  Theorem \ref{fffbispegaard}
 (normal form theorem for $\mathfrak{F}$).
If $\phi$ is  of  form $s\sqsubseteq t$,  then $\phi$ is $\mathfrak{B}$-equivalent a formula of the form $\exists x_1x_2[  x_1sx_2=t  ]$.
If $\phi$ is of the form $s\not\sqsubseteq t$, then then Lemma \ref{tyggegummi} says that $\phi$ is $\mathfrak{B}$-equivalent to a formula of the form 
$\forall v_{1} \sqsubseteq t_1 \exists v_{2}\ldots v_{k}[  s = t  ]$. The remaining cases of the inductive proof are similar to their respective cases in the proof
of Theorem \ref{fffbispegaard}.
\qed
\end{proof}

 We have not been able to  find an interesting normal form  for purely existential formulas which is
stronger than the one for $\Sigma$-formulas in $\mathfrak{B}$. 
It follows form the results in Karhum\"aki et al.\ \cite{karh} 
that the relation $x\not\sqsubseteq^\mathfrak{B} y$ cannot be defined in the structure $\mathfrak{B}^-$:
By Theorem 16 in \cite{karh}, it is not possible to define the language
$$
L \;\; = \;\; \{ \ \alpha \mid \mbox{$\alpha \in \buni$ and $\boldsymbol{1}\boldsymbol{0}\boldsymbol{1} \not\sqsubseteq^\mathfrak{B} \alpha$} \ \}
$$
by a word equation. 
If we could define  $x\not\sqsubseteq^\mathfrak{B} y$  in  $\mathfrak{B}^-$, then it would have been possible to define $L$ by a word equation.
We refer the reader to the paper itself for further details as Theorem 16 is a rather involved statement.

Since it is not possible to define $x\not\sqsubseteq^\mathfrak{B} y$  in $\mathfrak{B}^-$, the normal form theorem for purely existential 
formulas in $\mathfrak{D}$
(Theorem \ref{bispegaard}), will for sure be false with respect to $\mathfrak{B}$. So will the normal form theorem 
for purely existential formulas in $\mathfrak{F}$
(Theorem \ref{fffbispegaard}).

\section{Decidability and Undecidability}

\subsection{Fragments}

Let us start to track the border between the decidable and the undecidable.
On the one hand, we have Makanin's result (Theorem \ref{baaaaasigsafcomp}). We know that it is decidable if a sentence
of the form $\exists \vec{x}[s=t]$ holds in the  standard model. We also know that any purely existential formula is $\mathfrak{D}$-equivalent
to formula of this form. On the other hand it is not very hard to prove a stronger version of Theorem \ref{bsigsafcomp}:

\begin{theorem}[$\Sigma$-Definability of Computable Functions]\label{bsihhhomp}
For any (partially) computable function $f:\nat^n \rightarrow \nat$ there 
exists a $\Sigma$-formula $\phi(x_1,\ldots , x_n,y)$ such that
$$
f(a_1, \ldots, a_n)= b \;\; \Leftrightarrow 
\;\; \mathfrak{D} \models \phi(\overline{\boldsymbol{1}^{a_1}},\ldots , \overline{\boldsymbol{1}^{a_n}},\overline{\boldsymbol{1}^{b}})\; .
$$
Moreover, given a  definition of the function, we can compute the formula. 
\end{theorem}

With some effort, the interested reader should be able to accomplish a proof of this theorem.  The theorem implies that it is undecidable if a $\Sigma$-sentence
holds in $\mathfrak{D}$. It is easy to define the relations $\preceq^{\mathfrak{D}}$ and $\not\preceq^{\mathfrak{D}}$ by $\Sigma$-formulas in the structures 
$\mathfrak{B}$ and $\mathfrak{F}$, and
 the bounded quantifiers of $\mathfrak{D}$ can  be expressed by $\Sigma$-formulas in
in $\mathfrak{B}$ and $\mathfrak{F}$, e.g., if 
$$
\mathfrak{D}\models  \phi\;\;\; \Leftrightarrow \;\;\; \mathfrak{B}\models \phi' 
$$
then
$$
\mathfrak{D}\models  (\exists x \preceq t)\phi\;\;\; \Leftrightarrow \;\;\; \mathfrak{B}\models  (\exists xy \sqsubseteq t)[ \ xy=t \; \wedge \; \phi' \ ]
$$
Thus Theorem \ref{bsihhhomp}  also implies  that it undecidable if a $\Sigma$-sentence
holds in $\mathfrak{B}$ or $\mathfrak{F}$.

In order
to gain  further insight into what we can---and cannot---decide in bit theory, we need to keep track of 
 the number and the type of   quantifiers that  appear in our $\Sigma$-formulas.
We will say that a $\Sigma$-formula is
 a $\sigf{n}{m}{k}$-formula if it contains $n$ unbounded existential quantifiers, $m$ bounded existential quantifiers and $k$ bounded universal quantifiers. 
The {\em fragment} $\msigf{n}{m}{k}{A}$ is the set of $\sigf{n}{m}{k}$-sentences that are true in the
 $\mathcal{L}_{BT}$-structure  $\mathfrak{A}$. The {\em (purely existential) fragment} $\exists^{\mathfrak{A}}$ is the set of purely existential sentences that are true in the 
 $\mathcal{L}_{BT}$-structure  $\mathfrak{A}$ (recall that a purely existential formula is a $\Sigma$-formula with no occurrences of bounded universal quantifiers).
A $\Delta$-formula is a $\Sigma$-formula that contain no unbounded existential quantifiers, and the {\em fragment} $\Delta^{\mathfrak{A}}$ is the set of $\Delta$-sentences that are true in the 
 $\mathcal{L}_{BT}$-structure  $\mathfrak{A}$.
Note that 
$$
\exists^{\mathfrak{A}} \,  = \!\! \bigcup_{n,m\in \nat} \! \msigf{n}{m}{0}{A}\;\;\;\;\;\;\;\; \mbox{ and } \;\;\;\;\;\;\;\; \Delta^{\mathfrak{A}} \,  = \!\! \bigcup_{m,k\in\nat}\! \msigf{0}{m}{k}{A}\; .
$$

\subsection{Decidable Fragments}

\begin{theorem}
The fragments $\Delta^{\mathfrak{D}}$, $\Delta^{\mathfrak{B}}$ and $\Delta^{\mathfrak{F}}$ are decidable.
\end{theorem}

\begin{proof}
We prove that $\Delta^{\mathfrak{B}}$ is decidable. Let $\phi$ be a $\Delta$-formula.  The negation of a $\Delta$-formula
is logically equivalent to a $\Delta$-formula (use De Morgan's laws).
We can compute a  $\Delta$-formula $\phi'$ which is logically equivalent to $\neg \phi$.
 By Theorem \ref{bsigcomp},
we have  $B\vdash \phi$ if $\mathfrak{B}\models \phi$. By the same theorem, we have $B\vdash \phi'$ if $\mathfrak{B}\models \neg \phi$. 
The set of formulas derivable from
the axioms of $B$ is computably enumerable. Hence it is decidable if  $\phi$ is true in $\mathfrak{B}$. 

We can prove that the fragments $\Delta^{\mathfrak{D}}$ and  $\Delta^{\mathfrak{F}}$ are decidable in the same way
as we also have $\Sigma$-complete axiomatizations for both $\mathfrak{D}$ and $\mathfrak{F}$.
\qed
\end{proof}

In some sense, we kill a fly with a hammer when we use  $\Sigma$-completeness to prove the preceding theorem.
It can of course be checked by  brute-force algorithms if  $\Delta$-sentences are true in  the structures 
$\mathfrak{D}$, $\mathfrak{B}$ and $\mathfrak{F}$.

\begin{theorem}
The fragment  $\exists^{\mathfrak{D}}$ is decidable.
\end{theorem}

\begin{proof}
Theorem \ref{bispegaard} states that any purely existential formula $\phi$ is $\mathfrak{D}$-equivalent to a formula of
the normal form $\exists v_1\ldots  v_{m} [s=t]$. Our proofs show that there is an algorithm for transforming $\phi$ into this
normal form. Thus the theorem follows by Makanin's result (Theorem \ref{baaaaasigsafcomp}).
\qed
\end{proof}

\paragraph{Open Problem:} Is the fragment  $\exists^{\mathfrak{B}}$ decidable?

\paragraph{Open Problem:} Is the fragment  $\exists^{\mathfrak{F}}$ decidable?

\subsection{The Modulo Problem}

\newcommand{\mmm}{\mathcal{M}}
\newcommand{\X}{{\texttt{X}}}

In \cite{cie18} we use the Post's Correspondence Problem (Post \cite{post}) to prove that the fragments
 $\msigf{3}{0}{2}{D}$, $\msigf{4}{1}{1}{D}$, $\msigf{1}{2}{1}{B}$ and $\msigf{1}{0}{2}{B}$ are undecidable.
We will improve these results considerably in the next section. In this section 
we introduce an undecidable problem that is especially tailored for our needs in that respect, that is, an undecidable
problem that makes the proofs  in next  section smooth and transparent. 
In lack of a better name, we
dub the problem the {\em Modulo Problem}.

We assume that the reader is familiar with modulo arithmetic
and elementary number theory. We use $f^N$ to denote the $N^{\mbox{{\scriptsize th}}}$ iteration of an unary function $f$,
that is, $f^0(X)=X$ and $f^{N+1}(x)= f(f^N(x))$.

\begin{definition}
{\em The Modulo Problem} is given by
\begin{itemize}
\item Instance: a list of pairs $\langle A_{0}, B_{0} \rangle,\ldots , \langle A_{M-1}, B_{M-1} \rangle$ where 
$M>1$ and $A_i,B_i \in \nat$ (for $i=0,\ldots, M-1$).
\item Solution:  a natural number $N$ such that $f^N(3)=2$ where
$$
f(x) \; = \; A_jz+ B_j
 $$
if there exists $j<M$ such that $x=zM+j$.
\end{itemize}
\end{definition}

The undecidability of the Modulo Problem follows from the existence of a certain type of Collatz functions constructed in
Kurtz \& Simon \cite{kurtz}. We will spend the rest of this section to explain why the problem is undecidable.

We will work with the counter machines  introduced by Minsky \cite{minsky}. 
These machines are also known as register machines or Minsky machines. 
A {\em counter machine}  consists of
\begin{itemize}
\item a finite number of registers $\X_1,\ldots \X_n$
\item a finite number of instructions $I_1,\ldots , I_m$.
\end{itemize}
The registers store natural numbers. The instructions tell the machine  how to manipulate  these numbers. The machine starts by executing the instruction $I_2$. Each resister stores
0 when the execution starts (our counter machines do not take input).
The instruction $I_1$ is the unique halting instruction. This instruction does not modify any register, and the machine  halts if and only if this instruction is executed.
Each of the the remaining instructions $I_2,\ldots , I_m$ will either be an {\em increment instruction} or a {\em decrement instruction}.
An increment instruction $I_i$ is of the form $I_i: \X_j, I_k$. The instruction tells the machine to 
increment the natural number stored in the register $\X_j$ by one and then proceed with instruction $I_k$. A decrement instruction $I_i$ is of the form
$I_i :  \X_j, I_k, I_{\ell}$.
The instruction
tells the machine to decrement the natural number stored in the register $\X_j$ by one  if  this is possible, that is, if the number
is strictly greater than zero. 
If it is possible to decrement the number,  the machine will proceed with instruction $I_k$; if it is not possible, 
 the machine will   proceed with instruction $I_\ell$.  We  assume that we have $i\neq k$ in any increment instruction and 
$i\not\in \{ k ,\ell \}$ in any decrement instruction (this makes it easier to see that some of the arguments below indeed are correct).

It is well known that it is undecidable if a counter machine that starts with every register set to zero
will terminate with every register set to zero. Kurtz \& Simon \cite{kurtz} simulate counter machines 
by FracTran programs.
A {\em FracTran program} $f$ is a finite sequence $q_1,\ldots , q_s$ of  non-negative rational numbers.
The {\em transition function} $F_f$ for the FracTran program $f$ maps the natural number $x$ to $q_i x$ where
$i$ is the least $i$ such that $q_i x$ is integral. If no such $i$ exists, we regard $F_f(x)$ as undefined.

Following Kurtz \& Simon \cite{kurtz}, we will  show how we given a counter machine $\mmm$ can construct a FracTran program $f_\mmm$ such that 
there exists $N$ such that $F_{f_\mmm}^N(3)=2$ iff $\mmm$ terminates with every register set to zero.
Let $p_i$ denote the $i^{\mbox{{\scriptsize th}}}$ prime ($p_1=2$).

The execution of a counter machine can be viewed as a sequence
of {\em configurations} of the form
$$
I_i, a_1, a_2,\ldots , a_n
$$
where $I_i$ is the  instruction to be executed next and $ a_1, a_2,\ldots , a_n$, respectively, are the numbers stored in the registers $\X_1,\ldots \X_n$.
We represent a configuration  by a natural number of the form
$$
p_ip_{m+1}^{a_1}p_{m+2}^{a_2}\ldots p_{m+n}^{a_n}\; .
$$
where $i\leq m$.
Let $f_\mmm$ be a FracTran program $q_1,\ldots ,q_s$ such that
\begin{itemize}
\item[(1)] for each increment instruction $I_i: \X_j, I_k$ of $\mmm$'s, there is $t$ such that $q_t = p_k p_{m+j}/p_i$
\item[(2)] for each decrement instruction $I_i :  \X_j, I_k, I_{\ell}$ of $\mmm$'s, there is $t$ such that $q_t = p_k /p_ip_{m+j}$ 
and $q_{t+1} = p_\ell /p_i$
\item[(3)] $q_1,\ldots ,q_s$ is a minimal sequence that satisfies (1) and (2) (so no sequence of length $s-1$ will satisfy (1) and (2)).
\end{itemize}

Let us make some observations.
Let $I_i$ be  the increment instruction $I_i: \X_j, I_k$. Then $p_k p_{m+j}/p_i$ is the one and only $q_t$ in the sequence $q_1,\ldots ,q_s$
such that $q_t p_ip_{m+1}^{a_1}\ldots  p_{m+n}^{a_n}$ is integral.
Thus, we have
\begin{multline*}
F_{f_\mmm}(p_ip_{m+1}^{a_1}\ldots p_{m+j}^{a_j}\ldots p_{m+n}^{a_n}) \;\; = \;\; \frac{p_kp_{m+j}}{p_i} p_ip_{m+1}^{a_1}\ldots p_{m+j}^{a_j}\ldots p_{m+n}^{a_n} \\
\;\; = \;\;  p_kp_{m+1}^{a_1}\ldots p_{m+j}^{a_j + 1}\ldots p_{m+n}^{a_n}.
\end{multline*}
Let $I_i$ be  the decrement instruction $I_i :  \X_j, I_k, I_{\ell}$ and assume that  $\X_j$ stores $a_j>0$.
Now there are exactly two rationals  $r$ in the sequence $q_1,\ldots ,q_s$ such that $r p_ip_{m+1}^{a_1}\ldots  p_{m+n}^{a_n}$ is integral.
These two rationals are $p_k /p_ip_{m+j}$ and $p_\ell /p_i$. Since $p_k /p_ip_{m+j}$ occur before $p_\ell /p_i$,
 we have
\begin{multline*}
F_{f_\mmm}(p_ip_{m+1}^{a_1}\ldots p_{m+j}^{a_j}\ldots p_{m+n}^{a_n}) \;\; = \;\; \frac{p_k}{p_ip_{m+j}} p_ip_{m+1}^{a_1}\ldots p_{m+j}^{a_j}\ldots p_{m+n}^{a_n}
\\ \;\; = \;\;
p_kp_{m+1}^{a_1}\ldots p_{m+j}^{a_j -1 }\ldots p_{m+n}^{a_n} \; .
\end{multline*}
If $I_i$ is  the decrement instruction 
$I_i :  \X_j, I_k, I_{\ell}$ and   $\X_j$ stores $0$, then $p_\ell /p_i$ is the only $q_t$ in the sequence 
that makes $q_t p_ip_{m+1}^{a_1}\ldots  p_{m+n}^{a_n}$  integral, and we have
$$F_{f_\mmm}(p_ip_{m+1}^{a_1}\ldots  p_{m+n}^{a_n}) \;\; = \;\;
p_\ell p_{m+1}^{a_1}\ldots p_{m+n}^{a_n}\; . $$

Any counter machine starts in the configuration $I_2,0,0,\ldots,0$ 
(our counter machines do not take input and every register stores zero when
the execution starts). This start configuration is  represented by the prime $p_2$, that is, 3.
If  a counter machine  halts with  every register set to 0,  it halts in the configuration $I_1,0,0,\ldots,0$. 
This configuration is represented by
the prime $p_1$, that is, 2. Given the observations above, it should be  obvious that there exists $N$ 
such that $F^N_{f_\mmm}(3) =2$ iff $\mmm$ halts with every register  set to 0. Thus, we conclude that it is undecidable
if there exists $N$ such that $F^N_{f_\mmm}(3) =2$.

A Collatz function is a function $C:\nat\rightarrow \nat$ that can be written of the form
$$
C(x)  =  a_i x + b_i \;\;\;   \mbox{ if $x\equiv i \mod{M}$} 
$$
where $a_i,b_i\in \rational$ and $M\in\nat$ (for more on Collatz functions and the related Collatz problem, see \cite{kurtz}). 
Following the ideas of  Kurtz \& Simon \cite{kurtz},
we will now  construct a Collatz function $C_\mmm$ of the more restricted form
$$
C_\mmm(x)  =  a_i x  \;\;\;   \mbox{ if $x\equiv i \mod{M}$.} 
$$
such that we have $C_\mmm(x) = F_{f_\mmm}(x)$ whenever  $F_{f_\mmm}(x)$ is defined.
We  use the FracTran program  $f_\mmm = q_1, \ldots ,q_s$ to determine $M$ and   $a_0,a_1,\ldots ,a_{M-1}$.

For each $q_i$ in the FracTran program $f_\mmm = q_1, \ldots ,q_s$, 
let $c_i,d_i\in \nat$ be such that $c_i/d_i = q_i$ and $c_i$ and $d_i$ are relatively prime.
Set $M$ to the least common multiple of $d_1,\ldots, d_s$ and then follow the procedure \texttt{(I)} 
to determine  $a_0,a_1,\ldots ,a_{M-1}$. When the procedure starts every $a_j$ is undefined.
$$
\begin{array}{l}
\mbox{\texttt{PROCEDURE (I):}}\\
\mbox{for $i:=1,\ldots,s$ do}\\
\;\;\;\;\;\; \mbox{ for $j:=0,\ldots, M-1$ do}\\
\;\;\;\;\;\;\;\;\;\;\;\;\;\;\;\;\;\;        a_j := \begin{cases}
a_j & \mbox{if $a_j$ already is defined}\\
q_i & \mbox{if $a_j$ is not yet defined and $d_i$ divides $j$} \\
\mbox{undefined} & \mbox{otherwise.}
             \end{cases}
\end{array}$$

This completes the construction of $C_\mmm$. Many of the rationals $a_0,a_1,\ldots ,a_{M-1}$ might
still not be defined when the procedure terminates, but that is not important to us. 
The following claim is what  matters to us:
\begin{align*}
 F_{f_\mmm}(x)= q_i x \;\; \Leftrightarrow \;\; C_\mmm(x)=q_ix \; . \tag{claim}
\end{align*}

In order to see that our claim holds, assume that 
$ F_{f_\mmm}(x)= q_i x$. Then, we know that $q_ix\in\nat$  and that $q_\ell x\not\in\nat$ when $\ell < i$.
Let $x= zM + j$ where $j<M$ and let $c_i/d_i = q_i$ where $c_i$ and $d_i$ are relatively prime.
Then, we have
$$
 q_i x \; = \; \frac{c_i}{d_i}(zM + j) \; \in \; \nat \; .
$$
As $d_i$ divides $M$ but not $c_i$, it must be the case that $d_i$ divides $j$.
A symmetric argument yields that $d_\ell$ does not divide  $j$ when $\ell < i$.
Thus, our procedure sets $a_j$ to $q_i$. Hence, we have
$C_\mmm(x) = C_\mmm(zM+ j) = a_jx = q_ix$. This proves that the left-right implication of the claim holds.
It is easy to see that also the right-left implication holds.

It follows from the claim that 
 $F^N_{f_\mmm}(3) =2$ iff $C_\mmm^N(3)=2$. Thus, we conclude that it is undecidable if there exists  $N$ such that $C_\mmm^N(3)=2$
(since we already have concluded that it is  undecidable if there exists  $N$ such that $F^N_{f_\mmm}(3) =2$).

We are soon ready to conclude that the Modulo Problem is undecidable. The function $f$ in the Modulo Problem is a
total function. 
Now, $C_\mmm$ is by no means a total function, but if $C_\mmm(x)$ is defined and $x= zM+j$ where $j<M$, then it turns out that
we have $A_j,B_j\in\nat$ such that $C_\mmm(x)= A_jz + B_j$. To see that such $A_j$ and $B_j$ exist, observe that
$$
C_\mmm(zM+j) \; = \; \frac{c}{d}(zM+j) \; \in \; \nat
$$
for some relatively prime natural numbers $c$ and $d$. Moreover, the algorithm for constructing $C_\mmm$ assures that
$d$ divides both $M$ and $j$.
Hence, let $A_j= cM/d$ and let $B_j= cj/d$, and $A_j$ and $B_j$ will be natural numbers such that $C_\mmm(x)= A_jz+B_j$.
This entails that the following procedure will construct a sequence $\langle A_0,B_0\rangle, \ldots \langle A_{M-1},B_{M-1}\rangle$
of pair of natural numbers: 
$$
\begin{array}{l}
\mbox{\texttt{PROCEDURE (II):}} \\
 \mbox{for $j:=0,\ldots, M-1$ do}\\
\;\;\;\;\;\;\;  \mbox{ -- set  $A_j =  cM/d$ and   $B_j= cj/d$    if procedure \texttt{(I)} defines $a_j$ to equal $c/d$ } \\
\;\;\;\;\;\;\;  \mbox{ -- set  $A_j=B_j=0$  
if procedure \texttt{(I)} leaves $a_j$ undefined.}\\
\end{array}$$
Now,  let
$$
f(x) \; = \; A_jz+ B_j
 $$
if there exists $j<M$ such that $x=zM+j$. Then, we obviously have $f(x) = C_\mmm(x)$ whenever $C_\mmm(x)$ is defined.
Thus, as it is undecidable if there exists $N$ such that $C_\mmm^N(3)=2$, it is also undecidable if there exists $N$ such that $f^N(3)=2$.
These considerations should make it clear that the Modulo Problem is undecidable.

\subsection{Undecidable Fragments}

The stage is now set for our undecidability results. We know that the Modulo Problem is undecidable.
Let $\langle A_{0}, B_{0} \rangle,\ldots , \langle A_{M-1}, B_{M-1} \rangle$ be an instance of the problem,
and let 
$$
f(x) \; = \; A_jz+ B_j
 $$
if there exists $j<M$ such that $x=zM+j$. It is easy to see that there is $N$ such that $f^N(3)=2$ if and only
if there exists a bit string  of the form
\begin{align*}
  \boldsymbol{0}\boldsymbol{1}^{a_0+1}\boldsymbol{0}\boldsymbol{1}^{a_1+1}\boldsymbol{0}\ldots \boldsymbol{0}\boldsymbol{1}^{a_N +1}\boldsymbol{0} \tag{*}
\end{align*}
where
\begin{itemize}
\item $a_0=3$ and $a_N=2$
\item for each $i\in\{0,\ldots , N-1\}$ there is $z\in \nat$ and $j<M$ such that
$$
a_i = zM + j \; \mbox{ and } \; a_{i+1}=  A_jz + B_j \; .
$$
\end{itemize}

The challenge  is to claim the existence  of a bit string $x$ of the form (*) by using as few quantifiers as possible. We will of course need
one unbounded existential quantifier to claim the existence of $x$. Then we will need some quantifiers to state that $x$ is of the
desired form. In the structure $\mathfrak{D}$, we will state
that any prefix of $x$ of the form $y\boldsymbol{0}\boldsymbol{1}$ can be extended to a prefix of $x$ of the form
\begin{align*}
y\boldsymbol{0}\boldsymbol{1}\underbrace{\boldsymbol{1}\boldsymbol{1}\boldsymbol{1}\ldots \boldsymbol{1}\boldsymbol{1}\boldsymbol{1}}_{\mbox{\scriptsize $kM$ copies}}
\underbrace{\boldsymbol{1}\boldsymbol{1}\boldsymbol{1} \ldots \boldsymbol{1}\boldsymbol{1}\boldsymbol{1}}_{\mbox{\scriptsize $j$ copies}}
\boldsymbol{0}\boldsymbol{1}\underbrace{\boldsymbol{1}\boldsymbol{1}\boldsymbol{1} \ldots \boldsymbol{1}\boldsymbol{1}\boldsymbol{1}}_{\mbox{\scriptsize $kA_j$ copies}}
\underbrace{\boldsymbol{1}\boldsymbol{1}\boldsymbol{1} \ldots \boldsymbol{1}\boldsymbol{1}\boldsymbol{1}}_{\mbox{\scriptsize $B_j$ copies}}
\boldsymbol{0} \tag{**}
\end{align*}
for some $j$ and $k$. Recall that $j$, $M$, $A_j$, $B_j$ are fixed natural numbers. Thus, for each $j\in \{0 ,\ldots , M-1\}$,
we can express that $x$ has a prefix of the form (**) by a formula $\exists z [\psi_j(x,y,z)]$ where
$$
\psi_j(x,y,z) \;\; \equiv \;\; z1=1z \; \wedge \; y01\underbrace{zz \ldots zz}_{M}\underbrace{11 \ldots 11}_{j}01\underbrace{zz \ldots zz}_{A_j}\underbrace{11 \ldots 11}_{B_j}0 \; \preceq \; x \; .
$$
In addition we have to state that $x$ should start with $\boldsymbol{0}\boldsymbol{1}\boldsymbol{1}\boldsymbol{1}\boldsymbol{1}\boldsymbol{0}$
and end with $\boldsymbol{0}\boldsymbol{1}\boldsymbol{1}\boldsymbol{1}\boldsymbol{0}$.
This explains why the formula
\begin{multline*}
(\exists x)\Big[ \ 011110\preceq x \; \wedge \; 
(\forall y \preceq x)\big[  \  y01 \not\preceq x \; \vee \; y01110 = x \\ \vee  
(\exists z)[ \  \bigvee_{j=0}^{M-1} \psi_j(x,y,z) \ ]
     \  \big]            \  \Big]
\end{multline*}
claims the existence of a bit string of the form (*). The formula contains two unbounded existential quantifiers and one bounded universal quantifier.
If we could decide if such a formula is true in $\mathfrak{D}$, then we could decide the Modulo Problem. Hence, $\msigf{2}{0}{1}{D}$ is an undecidable fragment.

Similar reasoning will show that the fragments $\msigf{1}{0}{1}{B}$ and $\msigf{1}{3}{2}{F}$   are  undecidable. The details can be found below.

\begin{theorem}\label{milleifiredager}
The fragments $\msigf{2}{0}{1}{D}$, $\msigf{1}{0}{1}{B}$ and $\msigf{1}{3}{2}{F}$   are  undecidable.
\end{theorem}

\begin{proof}
For any $\mathcal{L}_{BT}$-term $t$ let   $\repnot{t}{0}\equiv e$ and 
let $\repnot{t}{n+1}\equiv  t \circ \repnot{t}{n}$. Furthermore, note that the
word equation $x\boldsymbol{1}= \boldsymbol{1}x$ is satisfied iff $x\in \{ \boldsymbol{1}\}^*$.

Let 
$\langle A_{0}, B_{0} \rangle,\ldots , \langle A_{M-1}, B_{M-1} \rangle$ be an instance of the Modulo Problem.

First we prove that $\msigf{2}{0}{1}{D}$ is undecidable.
For $i\in \{0, \ldots, M-1\}$, let
$$
\psi_j(x,y,z) \;\; \equiv \;\;    z1= 1z \; \wedge \;  y01\repnot{z}{M}\repnot{1}{j}01\repnot{z}{A_j}\repnot{1}{B_j}0\preceq x \; .
$$
and let $\phi$ be the $\sigf{2}{0}{1}$-formula
\begin{multline*}
(\exists x)\Big[ \ 01\repnot{1}{3}0\preceq x \; \wedge \; 
(\forall y \preceq x)\big[  \  y01 \not\preceq x \; \vee \; y01\repnot{1}{2}0 = x \\ \vee  
(\exists z)[ \  \bigvee_{j=0}^{M-1} \psi_j(x,y,z) \ ]
     \  \big]            \  \Big]
\end{multline*}
Then $\mathfrak{D}\models \phi$ if and only if the instance has a solution.
Obviously, there is an algorithm for constructing $\phi$ when the instance is given, and thus 
no algorithm can decide if a $\sigf{2}{0}{1}$-sentence is true in $\mathfrak{D}$, that is, 
the fragment $\msigf{2}{0}{1}{D}$ is undecidable.

Next we prove that $\msigf{1}{0}{1}{B}$ is undecidable.
For $j\in \{0,\ldots , M-1 \}$, let $\xi_j(y,x)$ be the formula 
\begin{multline*}
 \Big( \ y1=1y \; \wedge \; 01\repnot{y}{M}\repnot{1}{j}0\sqsubseteq x \; \wedge \; 
\repnot{y}{M}\repnot{1}{j}\neq \repnot{1}{2} \ \Big)   \;\; \rightarrow \;\; \\  01\repnot{y}{M}\repnot{1}{j}01\repnot{y}{A_j}\repnot{1}{B_j}0\sqsubseteq x \; .
\end{multline*}
Note that $\xi_j(y,x)$ is
 trivially equivalent to a  $\sigf{0}{0}{0}$-formula $\hat{\xi}_j(y,x)$.  Let $\phi'$ be the $\sigf{1}{0}{1}$-formula
$$
(\exists x ) \Big[ \ 01\repnot{1}{3}0 \sqsubseteq x \; \wedge \; (\forall y\sqsubseteq x)[ \ \bigvee_{j=0}^{M-1} \hat{\xi}_j(y,x) \ ] \ \Big]\; .
$$
Then $\mathfrak{B}\models \phi'$ if and only if the instance has a solution. There is an algorithm for constructing $\phi'$ from
the instance, and thus we conclude that $\msigf{1}{0}{1}{B}$ is undecidable.

We are left to prove that $\msigf{1}{3}{2}{F}$ is undecidable.
The following formula does to job:
\begin{multline*}
(\exists x) \Big[ \ (\exists v \sss x) [ \ 01\repnot{1}{3}0v =x \ ] \; \wedge \; 
(\forall uv \sss x) \big[ \ u01v \neq x \; \vee \; \\ u01\repnot{1}{2}0 = x 
 \; \vee \; (\exists zy\sss x)[ \     \bigvee_{j=0}^{M-1} \eta_j(z,u,y,x)       \   ] \ \big] \ \Big]
\end{multline*}
where
$$\eta_j(z,u,y,x) \;\; \equiv \;\;  z1=1z \; \wedge \; u 01 \repnot{z}{M}\repnot{1}{j}01 \repnot{z}{A_j}\repnot{1}{B_j}0y=x
$$ 
for $j \in \{0,\ldots , M-1 \}$.
\qed
\end{proof}

\begin{corollary}\label{hodepine}
(i) It is undecidable whether a sentence of the form $$(\exists x_1)(\forall y\preceq x_1)(\exists x_2)\ldots (\exists x_n)\, s=t $$
is true in $\mathfrak{D}$. (ii) It is undecidable whether a sentence of the form $$(\exists x_1)(\forall y_1 y_2 \sqsubseteq x_1)(\exists x_2)\ldots (\exists x_n)\, s=t $$
is true in $\mathfrak{B}$. (iii) It is undecidable whether a sentence of the form $$(\exists x_1)(\forall y_1 y_2 \sss x_1)(\exists x_2)\ldots (\exists x_n)\, s=t $$
is true in $\mathfrak{F}$.
\end{corollary}

\begin{proof}
Consider the $\sigf{2}{0}{1}$-sentence  $\phi$ if the proof of the preceding theorem.
By 
Theorem \ref{bispegaard},  $\phi$ is $\mathfrak{D}$-equivalent to sentence of the form
$$(\exists x_1)(\forall y\preceq x_1)(\exists x_2)\ldots (\exists x_n)\, s=t \; .$$
This proves that (i) holds. Furthermore, we have 
$$\mathfrak{D} \models (\exists x_1)(\forall y\preceq x_1)(\exists x_2)\ldots (\exists x_n)\,  s=t $$
if and only if
$$\mathfrak{B} \models (\exists x_1)(\forall y_1 y_2 \sqsubseteq x_1)(\exists x_2)\ldots (\exists x_n)[ \ y_1\circ y_2\neq x_1 \; \vee \;  s=t \ ]$$
if and only if
$$\mathfrak{F} \models (\exists x_1)(\forall y_1 y_2 \sss x_1)(\exists x_2)\ldots (\exists x_n)[ \ y_1\circ y_2\neq x_1 \; \vee \;  s=t \ ] \; . $$
Thus, (ii) and (iii) hold by Lemma \ref{famanda} and Lemma \ref{samanda}.
\qed
\end{proof}

\end{document}